\documentclass[12pt]{article}

\usepackage[margin=1.25in]{geometry}

\usepackage{amssymb,amsmath}
\usepackage{amsthm}
\usepackage{hyperref}
\usepackage{mathrsfs}
\usepackage{textcomp}
\usepackage{enumerate}
\usepackage{framed}
\usepackage{ulem}
\usepackage{color,xcolor}

\usepackage{soul}
\soulregister\cite7 
\soulregister\citep7 
\soulregister\citet7 
\soulregister\ref7 
\soulregister\pageref7 

\hypersetup{
    colorlinks,%
    citecolor=blue,%
    filecolor=blue,%
    linkcolor=blue,%
    urlcolor=black
}

\newtheorem{theorem}{Theorem}[section]
\newtheorem{definition}{Definition}[section]

\newtheorem{remark}{Remark}[section]
\newtheorem{proposition}{Proposition}[section]
\newtheorem{corollary}{Corollary}[section]

\numberwithin{equation}{section}

\allowdisplaybreaks[3]
\makeatletter

\newdimen\bibspace
\renewenvironment{thebibliography}[1]{%
 \section*{\refname 
       \@mkboth{\MakeUppercase\refname}{\MakeUppercase\refname}}%
     \list{\@biblabel{\@arabic\c@enumiv}}%
          {\settowidth\labelwidth{\@biblabel{#1}}%
           \leftmargin\labelwidth
           \advance\leftmargin\labelsep
           \itemsep\bibspace
           \parsep\z@skip     %
           \@openbib@code
           \usecounter{enumiv}%
           \let\p@enumiv\@empty
           \renewcommand\theenumiv{\@arabic\c@enumiv}}%
     \sloppy\clubpenalty4000\widowpenalty4000%
     \sfcode`\.\@m}
    {\def\@noitemerr
      {\@latex@warning{Empty `thebibliography' environment}}%
     \endlist}

\makeatother

\newcommand{\va}{\varepsilon}           \newcommand{\ud}{\mathrm{d}}
\newcommand{\be}{\begin{equation}}      \newcommand{\ee}{\end{equation}}

            \newcommand{\si}{\sigma}
         \newcommand{\Sn}{\mathbb{S}^n}
       \newcommand{\rn}{\mathbb{R}^n}

\newcommand{\wdt}{\ensuremath{\widetilde}}

\def\Xint#1{\mathchoice
	{\XXint\displaystyle\textstyle{#1}}%
	{\XXint\textstyle\scriptstyle{#1}}%
	{\XXint\scriptstyle\scriptscriptstyle{#1}}%
	{\XXint\scriptscriptstyle\scriptscriptstyle{#1}}%
	\!\int}
\def\XXint#1#2#3{{\setbox0=\hbox{$#1{#2#3}{\int}$ }
		\vcenter{\hbox{$#2#3$ }}\kern-.6\wd0}}

\def\dashint{\Xint-}

\begin{document}
	
	\title{ Unified results of  compactness and existence for
		prescribing fractional $Q$-curvatures problem	}
	\author{ {\sc Yan Li}\,, {\sc Zhongwei Tang}\thanks{The research was supported by National Science Foundation of China(12071036,12126306)}\,, {\sc Heming Wang} and {\sc Ning Zhou} \\
		\small School of Mathematical Sciences, \\
		\small Laboratory of Mathematics and Complex Systems, MOE,\\
		\small Beijing Normal University, Beijing, 100875, P.R. of China}
	
	\date{}

	\maketitle

		\begin{abstract}
			In this paper we study the problem of prescribing fractional $Q$-curvature of order $2\si$ for a conformal metric on the standard sphere $\Sn$ with $\si\in (0,n/2)$ and $n\geq2$. Compactness and existence results are obtained in terms of the flatness order $\beta$  of the prescribed curvature function $K$. Making use  of integral representations and perturbation result, we develop a unified approach to obtain these results when $\beta\in [n-2\si,n)$ for all $\si\in (0,n/2)$.    This work generalizes the corresponding results of Jin-Li-Xiong		[Math. Ann. 369: 109--151, 2017] for $\beta\in (n-2\si,n)$.
		\end{abstract}


	\section{ Introduction} 
	
	Let $(\Sn, g_{0})$ be the standard sphere in $\mathbb{R}^{n+1}$.
	The prescribing fractional $Q$-curvature problem
	of order $2\si $ on $\Sn $
	can be described as:
	which function $K$ on $\Sn $ is the fractional $Q$-curvature of a
	metric $g$ on $\Sn $  conformally equivalent to $g_0?$
	If we denote $g=v^{4/(n-2\si )}g_{0},$
	this problem can be represented  as finding the solution of the following
	nonlinear equation with critical exponent:
	\be\label{main-eq}
	P_{\si }(v)=c(n, \si )
	K v^{\frac{n+2 \si }{n-2 \si }} \quad \text { on }\, \Sn ,
	\ee
	where $n\geq 2,$ $0<\si <n/2,$
	$c(n,\si )=\Gamma(\frac{n}{2}+\si )/\Gamma(\frac{n}{2}-\si ),$
	$\Gamma$ is the Gamma function,  $K$ is a function defined on $\Sn,$
and 	$P_{\si }$ is an intertwining operator of   $2\si $-order:
	$$
	P_{\si }=\frac{\Gamma(B+\frac{1}{2}+\si )}{\Gamma(B+\frac{1}{2}-\si )}, \quad B=\sqrt{-\Delta_{g_{0}}
		+\Big(\frac{n-1}{2}\Big)^{2}},
	$$
where $\Delta_{g_{0}}$ is the Laplace-Beltrami operator on $(\Sn, g_{0})$.
	The operator $P_{\si}$ can be viewed as the pull back operator of the fractional Laplacian $(-\Delta)^{\si }$ on $\mathbb{R}^{n}$
via the stereographic projection:
$$
(P_{\si }(v)) \circ F=|J_{F}|^{-\frac{n+2 \si }{2 n}}(-\Delta)^{\si }(|J_{F}|^{\frac{n-2 \si }{2 n}}(v \circ F)) \quad \text { for } \, v \in C^{2\si }(\Sn ),
$$
where $F$ is the inverse of the stereographic projection and $|J_F|=(\frac{2}{1+|x|^2})^n$
is the determinant of the Jacobian of $F$.
In addition, the Green's function of $P_{\si }$ is the spherical
Riesz potential, i.e.,
\be\label{gr}
P_{\si }^{-1}f(\xi)=c_{n,\si }
\int_{\Sn }\frac{f(\zeta)}{|\xi-\zeta|^{n-2\si }}\,\ud \mathrm{v o l}_{g_{0}}(\zeta)
\quad \text{ for }\, f\in L^{p}(\Sn ),
\ee
where $c_{n,\si }
=\frac{\Gamma(\frac{n-2\si }{2})}{2^{2\si }\pi^{n/2}\Gamma(\si )}$,
$p>1,$ and $|\cdot|$ is the Euclidean distance in $\mathbb{R}^{n+1}$.

Eq. \eqref{main-eq} has a variational structure and involves critical exponent because of the Sobolev embeddings. A natural function space for finding solutions is
 $H^\si(\Sn)$,  the $\si$-order fractional Sobolev space that consists of all functions $v \in L^2(\Sn)$ such that $(1-\Delta_{g_0})^{\si / 2} v \in L^2(\Sn)$, with the norm $\|v\|_{H^\si(\Sn)}:=\|(1-\Delta_{g_0})^{\si / 2} v \|_{L^2(\Sn)}$. The sharp Sobolev inequality on $\Sn$ (see Beckner \cite{bw}) asserts that
\begin{align}\label{sobolev}
\Big(\dashint_{\Sn }|v|^{\frac{2 n}{n-2 \si}} \,\ud  \mathrm{v o l}_{g_{0}}\Big)^{\frac{n-2 \si}{n}} \leq \frac{\Gamma(\frac{n}{2}-\si)}{\Gamma(\frac{n}{2}+\si)} \dashint_{\Sn } v P_\si(v) \,\ud  \mathrm{v o l}_{g_{0}} \quad \text { for }\, v \in H^\si(\Sn).
\end{align}
Due to the non-compactness of the embedding of $H^\si(\Sn)$ into $L^{2n/(n-2\si)}(\Sn)$,
the Euler functional associated to \eqref{main-eq} does not satisfy the Palais-Smale condition, which leads to the failure of the standard critical point theory. Moreover, beside the obvious necessary condition that $K$ be positive somewhere, there are topological obstructions of Kazdan-Warner type to  solve \eqref{main-eq} (see \cite{Xu,jlxm}).

Problem  \eqref{main-eq}  can be seen as the generalization of the classical Nirenberg problem: which function $K$ on $\Sn$ is the scalar curvature of a metric conformal to the standard one? This is equivalent to solving 	\begin{align*}
P_{1}w+1=-\Delta_{g_{0}} w+1=K e^{w} \quad \text { on }\, \mathbb{S}^{2},
\end{align*}
and
\be\label{eq:Nirenberg}
P_1v=-\Delta_{g_{0}} v+\frac{n(n-2)}{4}  v=\frac{n-2}{4(n-1)} K v^{\frac{n+2}{n-2}} \quad \text { on }\, \Sn ,~n \geq 3,
\ee
where $g=e^{2w}g_{0}$ and  $v=e^{\frac{n-2}{4} w}$. There has been vast literature on the Nirenberg problem and related ones and it would be
impossible to mention here all works in this area. One significant aspect most directly related to this paper is the fine analysis of blow-up  solutions or the compactness of
the solution set. These were studied in \cite{LPrescribing1995, LPrescribing1996,JLXOn2014,HanPrescribing1990,jlxm,ChangYangPrescribing1987,ChangGurskyYang1993,BahriCoronThe1991,SZPrescribed1996}. For more recent and further studies, see our  work  \cite{LTZOn2022,LTZCompactness2022} and related references therein.
	
Another stimulating situation is the study of higher orders and fractional order conformally invariant pseudo-differential operators $P_k^{g_{0}}$ on $(\Sn, g_0)$, which exist for all positive integers $k$ if $n$ is odd and for $k=\{1, \ldots, {n}/{2}\}$ if $n$ is even.  These
	operators defined on Riemannian manifolds have also been studied.
	For any Riemannian manifold $(M,g)$,
	let $R_{g}$ be the scalar curvature of $(M,g),$
	and the conformal Laplacian be defined as $P_{1}^{g}=-\Delta_{g}+\frac{n-2}{4(n-1)}R_{g}.$
	The Paneitz operator $P_{2}^{g}$ is
	another conformal invariant operator, which was discovered by Paneitz \cite{Pa}.
	Graham et al. \cite{GJMS} constructed a sequence of conformally
	covariant elliptic operators $\{P_{k}\}$
	on Riemannian manifolds for all positive integers $k$ if
	$n$ is odd, and for $k\in\{ 1,\ldots,n/2\}$  if $n$ is even,
	which are called GJMS operators.
	Juhl \cite{ju1,ju2}  found  an explicit formula and a recursive formula
	for GJMS operators and $Q$-curvatures (see also Fefferman and
	Graham \cite{FGJuhl2013}).
	Graham and Zworski \cite{GZ} presented a family
	of fractional order conformally invariant operators
	$P_{\si }^{g}$  of non-integer order
	$\si \in (0,n/2)$
	on the conformal infinity of asymptotically hyperbolic manifolds.
	In addition, Chang and Gonz\'alez \cite{CG}
	showed that the operator $P_{\si }^{g}$ with
	$\si \in (0,n/2)$ can  be defined as a Dirichlet-to-Neumann operator
	of a conformally compact
	Einstein manifold by using localization method in \cite{CS},
	they also provided some new interpretations and properties of those fractional
	operators and their associated fractional $Q$-curvatures.
	There are	many research  conducted on the
	fractional operators $P_{\si }^{g}$ and their
	associated fractional $Q$-curvature,
	for instance, see \cite{AC,Chti1,clz,Chti2,CROn2011,
		DMAPrescribingI2002,DMAPrescribingII2002,ERMountain2002,
		JLXOn2014,JLXOn2015,jlxm,LTZOn2022,LTZCompactness2022}.
	
	Directly related to our current work are some work on  blow up analysis, a priori estimates, and existence and compactness of solutions to \eqref{main-eq}.  	For $\si \in(0,1)$, Jin-Li-Xiong \cite{JLXOn2014, JLXOn2015}
	proved the existence of the solutions to \eqref{main-eq}
	and derived some compactness properties.  More precisely, thanks to a very subtle approach based on approximation
	of the solutions of \eqref{main-eq} by a blow-up subcritical method, they proved the existence of solutions for  \eqref{main-eq}.  	For $\si \in(0,n/2)$,
	Jin-Li-Xiong \cite{jlxm} developed a unified approach to establish blow up profiles,
	compactness and existence of positive solutions to \eqref{main-eq}  by making use of integral representations.  Their main hypothesis is the so-called flatness condition. Namely,  let $ K\in C^{1}(\Sn )$
	($K\in C^{1,1}(\Sn )$ if $0<\si \leq1/2$)
	be a positive function. We  say that $K$ satisfies the flatness condition $(K)_{\beta}$ for some $\beta>0$ if for 	each critical point $q_{0}$ of $K$, in some geodesic normal	coordinates $\{y_{1}, \ldots, y_{n}\}$ centered at $q_{0}$,
	there exists a small neighborhood $\mathscr{O}$ of $0$  such
	that
	\begin{align}\label{kt}
	K(y)=K(0)+Q_{(q_{0})}^{(\beta)}(y)+R_{(q_{0})}(y)
	\quad \text{ for all }\, y \text{ close to 0},
	\end{align}
	where $Q_{(q_{0})}^{(\beta)}$ satisfies
	$$
	Q_{(q_{0})}^{(\beta)}(\lambda y)=\lambda^{\beta}
	Q_{(q_{0})}^{(\beta)}(y), \quad \forall\, \lambda>0,\,
	y \in \mathbb{R}^{n}, \,Q_{(q_{0})}^{(\beta)} \in C^{[\beta]-1,1}(\mathbb{S}^{n-1}),
	$$
	$R_{(q_{0})}(y)$ is $C^{[\beta \mid-1,1}$ near 0 with $\lim _{y \rightarrow 0} \sum_{0 \leq|\alpha| \leq[\beta]}|\partial^{\alpha} R_{(q_{0})}(y)||y|^{-\beta+|\alpha|}=0$. However, they were only able to handle the case $\beta \in (n-2\si,n)$ in the flatness hypothesis.   When the flatness order of $K$ is $n-2\si $, the
	$L^{\infty} (\Sn)$ estimates of the solutions to \eqref{main-eq} fail, see \cite{jlxm} for more details.
	
	The flatness condition excludes some very interesting functions $K$. In fact, note that an important	class of functions, which is worth including in any results of existence for \eqref{main-eq}, are the Morse functions
	($C^2$ having only nondegenerate critical points). Such functions can be written in the form $(K)_{\beta}$ with $\beta=2$. This special flatness type condition $\beta=2$ has been applied to  obtain existence and compactness results, see  Li \cite{LPrescribing1996} for  $\sigma=1$ with $n=4$  in the classical  Nirenberg problem \eqref{eq:Nirenberg};  Djadli-Malchiodi-Ahmedou \cite{DMAPrescribingII2002} in Paneitz operator $\si=2$ with $n=6$; Li-Tang-Zhou \cite{LTZOn2022} in the half Laplacian  $\si=1/2$ with $n=3$    and  \cite{LTZCompactness2022} in $Q$-curvatures problems.

Recently, there have been some works devoted to the existence  results via studying the flatness condition effect, and those  mainly use the critical points at infinity techniques introduced by
Bahri–Coron \cite{BahriCoronThe1991,BC2}.   For $\si\in (0,1)$,	see   Abdelhedi-Chtioui-Hajaiej \cite{Chti1}  with $\beta\in (1,n-2\si]$;  Abdelhedi-Chtioui \cite{AC} for a non-degeneracy condition  $n=2$ and $\si=1/2$. For $\si=2$, see Chtioui-Bensouf-Al-Ghamdi1 \cite{CBA2015} with $\beta=n$;    Al-Ghamdi-Chtioui-Rigane \cite{ACR2013}   with $\beta\in [1,n-4)$;   Chtioui-Rigan \cite{CROn2011}   with $\beta\in [n-4,n)$.	However,  for higher order case  including  $\si\in (0,n/2)$,  there are still plenty of technical difficulties which demand new ideas.
%

Convinced that the nondegeneracy assumption would exclude some interesting class of functions $K$, we adopt the flatness hypothesis used in \cite{JLXOn2014,JLXOn2015,jlxm}. But again, in order to include all
plausible cases $\beta\in [n-2\si,n)$ with $\si\in (0,n/2)$, we need to develop a new line of attack with new ideas. This is essentially due to the structure of the multiple blow-up points, which is much more complicated than in the classical setting. Many new phenomena emerge.
More precisely, it turns out that the strong interaction between the bubbles, in the case $\beta \in (n-2\si,n)$, forces all blow-up points
to be single,  and  $\beta=n-2\si$  can present multiple
blow-up points and  there is a phenomenon of balance that is the interaction of two bubbles of the same order with respect to the self interaction.

Our goal in this paper is  to include a larger class of functions $K$ in the existence and compactness results for \eqref{main-eq}.  For this aim, we develop  a self-contained approach which enables us to include  the  case $\beta\in [n-2\si,n)$ for all $\si\in (0,n/2)$.  In order to state our results, we need	the following notations and assumptions.

	Suppose that $K(x)$  satisfies $(K)_{\beta}$ condition with $\beta \in [n-2\si,n)$, assume also
	\begin{align}\label{q1}
	|\nabla Q_{(q_{0})}^{(\beta)}(y)| \sim|y|^{\beta-1} \quad \text { for all }\, y \text { close to }\, 0,
	\end{align}
	\begin{align}\label{q2}
	\left(\begin{array}{l}
	\int_{\mathbb{R}^{n}} \nabla Q_{(q_{0})}^{(\beta)}(y+\xi)(1+|y|^{2})^{-n}\, \ud y \\\\
	\int_{\mathbb{R}^{n}} Q_{(q_{0})}^{(\beta)}(y+\xi) \frac{1-|y|^{2}}{1+|y|^{2}}
	(1+|y|^{2})^{-n} \,\ud y
	\end{array}\right) \neq 0, \quad \forall\, \xi \in \mathbb{R}^{n},
	\end{align}
	\begin{align}\label{q3}
	\left(\begin{array}{l}
	\int_{\mathbb{R}^{n}} \nabla Q_{(q_{0})}^{(\beta)}(y+\xi)(1+|y|^{2})^{-n}\, \ud y \\\\
	\int_{\mathbb{R}^{n}} Q_{(q_{0})}^{(\beta)}(y+\xi)(1+|y|^{2})^{-n}\, \ud y
	\end{array}\right) \neq 0, \quad \forall\, \xi \in \mathbb{R}^{n} .
	\end{align}
For any $\alpha \in [n-2\si,n)$,	define
	$$
	\mathscr{K}_{\alpha}=\{q_{0} \in \Sn : \nabla_{g_{0}} K(q_{0})=0, \,
	\beta(q_{0})=\alpha\},
	$$
where $\beta(q_{0})$ represents the flatness order of $K$ at the point $q_0$.
	For $q_{0} \in \mathscr{K}_{n-2\si }$, we assume
	\begin{align}\label{iq2}
	\int_{\mathbb{R}^{n}} \nabla Q_{(q_{0})}^{(n-2\si )}(y+\xi)
	(1+|y|^{2})^{-n} \, \ud y=0 \quad \text{ if and only if } \xi=0.
	\end{align}
	Let
	\begin{align}\label{kn-2s}
	\mathscr{K}_{n-2\si }^{-}=
	\Big\{q_{0} \in \mathscr{K}_{n-2\si }:
	\int_{\mathbb{R}^{n}} z \cdot \nabla Q_{(q_{0})}^{(n-2\si )}(z)(1+|z|^{2})^{-n} \, \ud z<0\Big\},
	\end{align}
	and for any distinct $q^{(1)},q^{(2)} \in \mathscr{K}_{n-2\si }^{-}$,
	$M=M(q^{(1)}, q^{(2)})$ is a symmetric $2 \times 2$ matrix given by
	\begin{align}\label{m}
	M_{i j}= \begin{cases}
	\displaystyle -K(q^{(j)})^{-\frac{1+\si }{\si }}
	\int_{\mathbb{R}^{n}} y \cdot \nabla Q_{q^{(j)}}^{(n-2\si )}(y)(1+|y|^{2})^{-n} \, \ud y, & i=j, \\
	\displaystyle -\frac{2^{\frac{n-2\si }{2}}(n-2\si )^2}{4n}
	\frac{\pi^{n/2}}{\Gamma(\si +\frac{n}{2})}\frac{G_{q^{(i)}}(q^{(j)})}{\sqrt{K(q^{(i)}) K(q^{(j)})}}, & i \neq j,
	\end{cases}
	\end{align}
	where
	\be\label{gf}
	G_{q^{(i)}}(q^{(j)})
	=\Big(\frac{1}{1-\cos d(q^{(i)},q^{(j)})}\Big)^{\frac{n-2\si }{2}}
	\ee
	is the Green's function of $P_{\si }$ on $\Sn ,$
	and  $d(\cdot\,,\,\cdot)$ denotes the geodesic distance.

	Let $\gamma\in(0,1)$, $C^{\gamma}(\Omega)$ denotes the standard
	H$\ddot{\mathrm{o}}$lder space over the domain $\Omega$.
	For simplicity, we use $C^{\gamma}(\Omega)$ to denote $C^{[\gamma],\gamma-[\gamma]}(\Omega)$
	when $1<\gamma\notin \mathbb{N}_{+}.$
	For $R>0$, $\alpha\in(0,1)$ and $\si \in (0,n/2)$,  we define
	$$
	\mathscr{O}_{R}:=\{v\in C^{2\si +\alpha}(\Sn ):
	1/R<v<R,\,\|v\|_{C^{2\si +\alpha}}<R\}.
	$$

	For $P \in \Sn $, $1 \leq t<\infty$,  let $\varphi_{P, t}$ be the M\"obius transformation on
	$\Sn$ which, under stereographic projection with respect to the north pole $P$, sends $y$ to $ty$ (see \cite{LPrescribing1995}). The totality of such a set of conformal transforms is diffeomorphic to the unit ball $B^{n+1}$ in $\mathbb{R}^{n+1}$, with the identity transformation identified with the origin in $B^{n+1}$ and
	\be\label{eq:varphiPt2}
	\varphi_{P, t} \leftrightarrow(\frac{t-1}{t}) P=: p \in B^{n+1}.
	\ee

	Our main result is:
		\begin{theorem}\label{them2}
		Let $K \in C^{1}(\Sn )$ ($K\in C^{1,1}(\Sn )$
		if $\si \leq 1/2$) be a positive function satisfying that for
		any critical point $q_{0} \in \Sn $ of $K$,
		there exists some real numbers
		$\beta=\beta(q_{0}) \in[n-2\si , n)$
		such that $K\in C^{[\beta],\beta-[\beta]}(\Sn )$ and \eqref{kt}--\eqref{q3} hold in some
		geodesic normal coordinate system centered at $q_{0}$.
		Suppose also that  if either $\sharp \mathscr{K}_{n-2\si }^{-} \leq 1$
		or for any two distinct points $q^{(1)}, q^{(2)} \in \mathscr{K}_{n-2\si }^{-}$,
      we have $M_{11} M_{22}<M_{12}^{2}$, where $M=M(q^{(1)},q^{(2)})$.		

		Then for all $0<\alpha,$ $\va<1$,   there exists some constant $C=C(n,\delta,\va,\alpha)$ such that,
		for all $\va\leq \mu\leq 1$, any positive solution $v$ to \eqref{main-eq}
		with $K$ replaced by $K_{\mu}=\mu K+(1-\mu)$, we have
		\begin{align}\label{th-bound}
		1 / C<v<C, \quad\|v\|_{C^{2\si +\alpha}(\Sn )}<C.
		\end{align}
		For all $ P \in \Sn $, $t \geq C$, we have
		\begin{align}\label{them-eq-1}
		\int_{\Sn } K \circ \varphi_{P, t}(x) x \neq 0,
		\end{align}
		and for all $R \geq C$, $t \geq C$,
		\begin{align}\label{them-eq-2}
		\operatorname{deg}(v-(P_{\si })^{-1} K v^{\frac{n+2\si }{n-2\si }}, \mathscr{O}_{R}, 0)
		=(-1)^{n} \operatorname{deg}\Big(\int_{\Sn } K \circ \varphi_{P, t}(x) x, B^{n+1}, 0\Big).
		\end{align}
		
		If we further assume that
		\begin{align*}
		\mathrm{deg}\Big( \int_{\Sn } K\circ \varphi_{P,t}(x)x, B^{n+1},0\Big) \ne 0 \quad \text{	for large $t$,}
		\end{align*}
		then \eqref{main-eq} has at least one solution.
	\end{theorem}

		\begin{remark}
	For $n \geq 3$ and $\si=1$, the above result was established by  Li \cite{LPrescribing1996}.	
Theorem \ref{them2}   gives the compactness and existence results of
the solution to \eqref{main-eq} when $K$ satisfies the  flatness condition
 $\beta \in [n-2\si,n)$.
Moreover,  Theorem \ref{them2} establishes a unified  result on the compactness and existence of solutions corresponding to prescribing fractional $Q$-curvatures problem.
		\end{remark}

	If we consider a more specific expression of $K(y)$ with $Q(y)=\sum_{j=1}^{n}a_{j}|y_{j}|^{\beta}$,
  additional  degree-counting formula of the solutions to \eqref{main-eq}  will be obtained. 	
	\begin{corollary}\label{cor1}
		Let  $K \in C^{1}(\Sn )$ ($K\in C^{1,1}(\Sn )$ if $0<\si \leq 1/2$) be a positive function satisfying that for any critical point $q_{0}$ of
		$K$,
		under the stereographic projection coordinate system $\{y_1,\ldots, y_{n}\}$
		with $q_{0}$ as south pole,
		there exist some small neighbourhood $\mathscr{O}$ of $0$
		and some real number $\beta \in[n-2\si ,n)$, such that
		$K\in C^{[\beta],\beta-[\beta]}(\Sn )$
		and
		$$
		K(y)=K(0)+\sum_{j=1}^{n}a_{j}|y_{j}|^{\beta}+R_{(q_0)}(y)\quad\text{ in }\,\mathscr{O},
		$$
		where $a_{j}=a_{j}(q_{0})\ne 0$, $\sum_{j=1}^n a_{j}\ne 0$, and
		$R_{(q_{0})}(y)\in C^{[\beta]-1,1}(\mathscr{O})$
		with $$\sum_{|\alpha|=0}^{[\beta]}|\partial^{\alpha} R_{(q_{0})}(y)||y|^{-\beta+|\alpha|}\to 0
		\quad \text{ as }\, |y|\to 0.$$
		
		Suppose also that if either $\sharp \mathscr{K}_{n-2\si }^{-} \leq 1$ or $M_{11} M_{22}<M_{12}^{2}$
		for all distinct $q^{(1)},$ $q^{(2)} \in \mathscr{K}_{n-2\si }^{-}$,
		where
		$\mathscr{K}_{n-2\si }^{-} $ is as in \eqref{kn-2s} and $M=M(q^{(1)}, q^{(2)})$
		is as in \eqref{m}.
		Then for all $0<\alpha<1$,  there exists some constant $C$ such that,
		for all solutions $v$ of \eqref{main-eq}, we have
		$$
		1 / C<v<C, \quad\|v\|_{C^{2\si +\alpha}(\Sn )}<C,
		$$
		and for all $R \geq C$,
		$$
		\operatorname{deg}(v-(P_{\si })^{-1} K v^{\frac{n+2\si }{n-2\si }}, \mathscr{O}_{R}, 0)
		=-1+(-1)^{n} \sum_{\nabla_{q_{0}} K(q_{0})=0 \atop
			\sum_{j=1}^{n} a_{j}(q_{0})<0}(-1)^{i(q_{0})},
		$$
		where
		$$
		i(q_{0})=\sharp\{a_{j}(q_{0}): a_{j}(q_{0})<0,\, 1 \leq j \leq n\}.
		$$
		If we further assume that
		$$
		\sum_{\nabla_{g_0} K(q_{0})=0 \atop  \sum_{j=1}^{n} a_{j}(q_{0})<0}(-1)^{i(q_{0})} \neq(-1)^{n},
		$$
		then \eqref{main-eq} has at least one solution.
	\end{corollary}
\begin{remark}
Corollary \ref{cor1} generalizes the corresponding results of Jin-Li-Xiong
	\cite[Theorem 1.4]{jlxm}.
\end{remark}	
	

Our methods rely on a readapted characterization of  blow up behavior  introduced by  Schoen-Zhang \cite{SZPrescribed1996,SZcv} and used in the above mentioned
papers \cite{LPrescribing1995,LPrescribing1996,JLXOn2014,JLXOn2015,jlxm}.
However, there is a serious problem of divergence of the integrals when $\beta=n-2\si$.
 To overcome this challenging problem, we perform a local analysis to give precise estimates to
 further characterize the blow up behavior.
In detail, we obtain the necessary conditions for the solution to \eqref{main-eq}
  blow up at more than one point by using the  blow up analysis,
  the Pohozaev type identity (see Proposition \ref{pro-pz}), and the assumptions of of $K$.
  This approach is different from the proof  in \cite{jlxm} with the case $n-2\si<\beta<n$.

The present paper is organized as the following.
	In Section \ref{sec:2}, we  characterize the blow up points for solutions
	to \eqref{main-eq}, which plays a key role in proving the compactness result of
	Theorem \ref{them2}  (see Theorem \ref{pro-lambda-exi}).
	The proof of Theorem \ref{pro-lambda-exi}  is
	mainly based on Pohozaev type identity (see Proposition \ref{pro-pz}) and the results of the blow up analysis established by Jin-Li-Xiong \cite{jlxm}.
	In Section \ref{sec:3},  we follow the arguments of Jin-Li-Xiong \cite{JLXOn2015} and   establish a perturbation result for all $\si \in(0,n/2)$ (see Theorem \ref{prop-theorem3.1}),
	which is necessary to prove the
	existence result of Theorem \ref{them2}.
	In Section \ref{sec:4}, we complete the proof of Theorem \ref{them2}
	and Corollary \ref{cor1} by using Theorem \ref{pro-lambda-exi},
	Theorem \ref{prop-theorem3.1} and some results in \cite[Section 6]{LPrescribing1995}.
		In Appendix \ref{sec:A}, we provide several technical results obtained in Jin-Li-Xiong \cite{jlxm}, which is necessary in our proof. 

	\section{Characterization of blow up behavior}\label{sec:2}
	
	In this section, we further characterizes the behavior of blow up points for solutions
	to \eqref{main-eq} by using integral representation, some blow up estimates in Appendix \ref{sec:A} and
 the properties of  matrix $M$, which  plays a key role in proving main result
 concerning compactness and existence.	
Theorem \ref{pro-lambda-exi} below also gives a necessary condition when the solution to \eqref{main-eq}
has more than one isolate simple blow up point.

We first review the definitions of  blow up point.
Let $\Omega$ be a domain in $\mathbb{R}^{n}$ and $K_{i}$ are
 nonnegative bounded functions in $\mathbb{R}^{n}.$
 Let $\{\tau_{i}\}_{i=1}^{\infty}$ be a sequence of
 nonnegative constants satisfying $\lim _{i \rightarrow \infty} \tau_{i}=0$, and set
$$
p_{i}=\frac{n+2 \sigma}{n-2 \sigma}-\tau_{i}.
$$
Suppose that $0 \leq u_{i} \in L_{{loc}}^{\infty}(\mathbb{R}^{n})$
satisfies the nonlinear integral equation
\be\label{2.1}
u_{i}(x)=\int_{\mathbb{R}^{n}} \frac{K_{i}(y) u_{i}(y)^{p_{i}}}{|x-y|^{n-2 \sigma}}\, \ud y
\quad \text { in } \,\Omega.
\ee
We assume that $K_{i} \in C^{1}(\Omega)$ $(K_{i}\in C^{1,1}(\Omega)$ if
$\sigma\leq 1/2$) and, for some positive constants $A_{1}$ and $A_{2}$,
\be\label{2.2111}
1 / A_{1} \leq K_{i}, \quad \text { and } \quad
\|K_{i}\|_{C^{1}(\Omega)} \leq A_{2},\,( \|K_{i}\|_{C^{1,1}(\Omega)} \leq A_{2}\, \text { if }\,\sigma \leq \frac{1}{2}).
\ee

\begin{definition}
  Suppose that $\{K_{i}\}$ satisfies \eqref{2.2111} and $\{u_{i}\}$ satisfies \eqref{2.1}.
  A point $\overline{y} \in \Omega$ is called a blow up point of $\{u_{i}\}$ if there exists a sequence
  $y_{i}$ tending to $\overline{y}$ such that $u_{i}(y_{i}) \rightarrow \infty$.
\end{definition}

\begin{definition}\label{defn1.1}
A blow up point $\overline{y} \in \Omega$ is called an isolated blow up point of $\{u_{i}\}$ if there exists $0<\overline{r}<\operatorname{dist}(\overline{y}, \Omega)$, $\overline{C}>0$, and $a$
sequence $y_{i}$ tending to $\overline{y}$, such that $y_{i}$ is a local maximum point of $u_{i},$ $u_{i}(y_{i}) \rightarrow \infty$ and
\be\label{1.57}
u_{i}(y) \leq \overline{C}|y-y_{i}|^{-2 \sigma /(p_{i}-1)} \quad \text { for all }\, y \in B_{\overline{r}}(y_{i}).
\ee
\end{definition}
Let $y_{i} \rightarrow \overline{y}$ be an isolated blow up point of $\{u_{i}\}$, and define,
for $r>0$,
$$
\overline{u}_{i}(r):=\frac{1}{|\partial B_{r}(y_{i})|} \int_{\partial B_{r}(y_{i})} u_{i}
\quad \text{and}\quad
\overline{w}_{i}(r):=r^{2 \sigma /(p_{i}-1)} \overline{u}_{i}(r).
$$

\begin{definition}\label{defn1.3}
A point  $y_{i} \rightarrow \overline{y} \in \Omega$ is called an isolated simple blow up point if $y_{i} \rightarrow \overline{y}$ is an isolated blow up point such that for some $\rho>0$ (independent of i), $\overline{w}_{i}$ has precisely one critical point in $(0, \rho)$ for large $i$.
\end{definition}

In what follows, we consider a situation more general
than the properties of set $\mathcal{K}_{n-2\si}^{-}$ given in \eqref{iq2} and \eqref{kn-2s}.
 Let $K \in C^1(\mathbb{S}^n)$ ($K\in C^{1,1}(\Sn )$
		if $\si \leq 1/2$)  be some positive function satisfying
that for any critical point $q_0 \in \mathbb{S}^n$ of $K$, there exists some real number $\beta=\beta\left(q_0\right) \in[n-2, n)$ such that \eqref{kt}--\eqref{q3} hold in some geodesic normal coordinate system centered at $q_0$.
Let $\widehat{\mathscr{K}}_{n-2\si}^{-}$ denote the set of critical points $q_0$ of $K$ with $\beta(q_0)=n-2\sigma$ and simultaneously for some $\eta_{0}\in \mathbb{R}^{n}$ satisfying
\begin{align}\label{gkn-2s}
\left\{\begin{array}{l}
\int_{\mathbb{R}^n} \nabla Q_{(q_0)}^{(n-2\si)}(y+\eta_0)(1+|y|^2)^{-n}\, \ud y=0, \\\\
\int_{\mathbb{R}^n} y \cdot \nabla Q_{(q_0)}^{(n-2\si)}(y+\eta_0)(1+|y|^2)^{-n}\, \ud y<0 .
\end{array}\right.
\end{align}

When $\# \widehat{\mathscr{K}}_{n-2\si}^{-} \geq 2$,
for distinct $q^{(1)}, \ldots, q^{(k)} \in \widehat{\mathscr{K}}_{n-2\si}^{-}$,
$\eta^{(j)} \in \mathbb{R}^n$ $(1 \leq j \leq k)$, satisfying \eqref{gkn-2s} with
$q_0=q^{(j)},$ $\eta_0=\eta^{(j)}$, we define a $k \times k$ symmetric metric $M=M\left(q^{(1)}, \ldots, q^{(k)}, \eta^{(1)}, \ldots, \eta^{(k)}\right)$ by
\begin{align}\label{mij}
	M_{i j}= \begin{cases}
	\displaystyle -K(q^{(j)})^{-\frac{1+\si }{\si }}
	\int_{\mathbb{R}^{n}} y \cdot \nabla Q_{q^{(j)}}^{(n-2\si )}(y+\eta^{(j)})(1+|y|^{2})^{-n} \, \ud y, & i=j, \\
	\displaystyle -\frac{2^{\frac{n-2\si }{2}}(n-2\si )^2}{4n}
	\frac{\pi^{n/2}}{\Gamma(\si +\frac{n}{2})}\frac{G_{q^{(i)}}(q^{(j)})}{\sqrt{K(q^{(i)}) K(q^{(j)})}}, & i \neq j.
	\end{cases}
	\end{align}

The result about characterization of blow up behavior
of the solutions to \eqref{main-eq} is:
	\begin{theorem}\label{pro-lambda-exi}
		Let  $K \in C^{1}(\Sn ) $ ($K\in C^{1,1}(\Sn )$
		if $\si \leq 1/2$) be a positive function satisfying that for any critical point $q_{0}$ of
		$K$, there exists some real number $\beta=\beta(q_{0}) \in[n-2\si , n)$,
		such that \eqref{kt}--\eqref{q3}
		hold in some geodesic normal coordinate system centered at $q_{0}$.
		Let  $\{v_{i}\}$ be a sequence of solutions to \eqref{main-eq}
		that blows up at $\{q^{(1)}, \ldots, q^{(k)}\}$ with $k \geq 2$.
		Then we have $q^{(1)}, \ldots, q^{(k)} \in \widehat{\mathscr{K}}_{n-2\si }^{-}$,
		and for some $\eta^{(j)} \in \mathbb{R}^{n}$ satisfying
 \eqref{gkn-2s}
		with $q_{0}=q^{(j)},$ $\eta_{0}=\eta^{(j)}(1 \leq j \leq k)$, the equation
		$$
		\sum_{\ell=1}^{k} M_{j \ell} \lambda_{\ell}=0
		$$
		has at least one solution $\lambda_{\ell} >0,$ $\ell=1, \ldots, k$,
		 where $\widehat{\mathscr{K}}_{n-2\si }^{-}$ is as in \eqref{gkn-2s} and $M_{j\ell}$
is as in \eqref{mij}.
	\end{theorem}
	
	\begin{proof}
		It follows from \cite[Theorem 3.3]{jlxm} that,
		after passing to a subsequence, $\{v_{i}\}$ has only isolated
		simple blow up points. Moreover,
		if $\{v_{i}\}$ blows up at $\{q^{(1)}, \ldots, q^{(k)}\}$ with $k \geq 2$,   we know from  \cite[Theorem 3.4]{jlxm} that $\beta(q^{(j)})=n-2\si $ for each $ j \in \{1,\ldots,k\}$.

		Using \eqref{gr}, we write \eqref{main-eq} as the form
		\begin{align}\label{ieq}
		v_{i}(\xi)=\frac{\Gamma(\frac{n+2 \si }{2})}{2^{2 \si } \pi^{n / 2} \Gamma(\si )}
		\int_{\Sn } \frac{K(\eta)
			v_{i}(\eta)^{\frac{n+2\si }{n-2\si }}}{|\xi-\eta|^{n-2\si }}\,
		\ud \eta \quad \text { on }\, \Sn .
		\end{align}
		Let $F$ be the stereographic projection	with $q^{(j)}$ being the south pole:
		\begin{align*}
		F:\mathbb{R}^n &\longrightarrow \Sn  \backslash \{-q^{(j)}\},\\
		x&\longmapsto\Big(\frac{2x}{1+|x|^2},\frac{|x|^2-1}{|x|^2+1}\Big),
		\end{align*}
		and its Jacobi determinant takes $|J_F|=(\frac{2}{1+|x|^2})^n$. Via the stereographic projection, the equation \eqref{ieq}
		is translated to
		$$
		u_{i}(x)=\frac{\Gamma(\frac{n+2 \si }{2})}{2^{2 \si } \pi^{n / 2} \Gamma(\si )}
		\int_{\mathbb{R}^n}
		\frac{\wdt{K}(y)u_{i}(y)^{\frac{n+2\si }{n-2\si }}}{|x-y|^{n-2\si }}\,\ud y
		\quad
		\text{ on }\, \mathbb{R}^{n},
		$$
		where
		\be\label{ste-p}	
		u_i(x)=H(x)v_i(F(x)), \quad \widetilde{K}(x)=K(F(x)),\quad H(x)=|J_F(x)|^{\frac{n-2 \si }{2 n}}=\Big(\frac{2}{1+|x|^2}\Big)^{\frac{n-2\si }{2}}.
		\ee
Let $x_{i}^{(j)}$ be the local maximum of $u_{i}$ and $x_{i}^{(j)}\rightarrow x^{(j)}=0.$
It follows from Propositions \ref{pro-isbp-upb} and \ref{pro-isbp-lim} that
		\begin{align}
		\begin{aligned}\label{uiu-lim}
		u_{i}(x_{i}^{(j)})u_{i}(x)\rightarrow h^{(j)}(x) :=a
		K(q^{(j)})^{\frac{2\si -n}{2\si }}|x|^{2\si -n}&+b^{(j)}(x)\\
		\quad& \text { in }\, C_{{loc}}^{2}(\mathbb{R}^{n}
		\backslash \cup_{\ell=1}^{k}\{x^{(\ell)}\}),
		\end{aligned}
		\end{align}
		where
		\begin{align}\label{uiu-lim-a}
		a=2^{n}c(n,\si )c_{n,\si }
		\int_{\mathbb{R}^{n}}\Big(\frac{1}{1+|y|^{2}}\Big)^{\frac{n+2\si }{2}}\, \ud y
		=
		2^{n-1}c(n,\si )c_{n,\si }B(\delta,n/2)=2^{n-2\si },
		\end{align}
		and
		$b^{(j)}(y)>0$ is some regular harmonic function in $\mathbb{R}^{n}\backslash \cup_{\ell\ne j}\{x^{(\ell)}\}$.
		Coming back to $v_{i}$ on $\Sn $, by \eqref{ste-p}, we have
		$$
		\lim_{i\to\infty}v_{i}(q_{i}^{(j)})v_{i}(q)=
		\lim_{i\to\infty}\Big(\frac{1+|x|^{2}}{2^{2}}\Big)^{\frac{n-2\si }{2}}
		u_{i}(x_{i}^{(j)})u_{i}(x).
		$$
		Thus, for $q\ne q^{(j)}$ and close to $q^{(j)}$,
		\begin{align*}
		\lim_{i\to\infty}v_{i}(q^{(j)}_{i})v_{i}(q)= a 2^{\frac{2\si -n}{2}}
		K(q^{(j)})^{\frac{2\si -n}{2\si }}
		G_{q^{(j)}}(q)+&\wdt{b}^{(j)}(q)\\
		&\quad \text{in } C_{loc}^{2}(\Sn \backslash
		\{q^{(1)},\ldots,q^{(k)}\}),
		\end{align*}
		where $G_{q^{(i)}}(q^{(j)})$ is as in \eqref{gf},
		and $\wdt{b}^{(j)}(q)$ is some regular function near $q^{(j)}$
		satisfying $P_{\si }\wdt{b}^{(j)}=0$.
		
		Then, taking into account the contribution of all
		the poles, we deduce that
		\begin{align}
		\lim_{i\to\infty}v_{i}(q^{(j)}_{i})v_{i}(q)
		= a 2^{\frac{2\si -n}{2}}
		\Big\{
		\frac{G_{q^{(j)}}(q)}{K(q^{(j)})^{\frac{n-2\si }{2\si }}}
		+\sum_{\ell\ne j}\lim_{i\to\infty}
		\frac{v_{i}(q^{(j)}_{i})}{v_{i}(q_{i}^{(\ell)})}
		\frac{G_{q^{(\ell)}}(q)}{K(q^{(\ell)})^{\frac{n-2\si }{2\si }}}
		\Big\}.
		\end{align}
		It follows that for $|x|>0$ small,
		\begin{align}
		&\lim_{i\to\infty}u_{i}(x_{i}^{(j)})
		u_{i}(x)\notag\\
		=&
		aK(q^{(j)})^{\frac{2\si -n}{2\si }}|x|^{2\si -n}
		+  a 2^{\frac{n-2\si }{2}}
		\sum_{\ell\ne j}\lim_{i\to\infty}
		\frac{v_{i}(q^{(j)}_{i})}{v_{i}(q_{i}^{(\ell)})}
		\frac{G_{q^{(\ell)}}(q^{(j)})}{K(q^{(\ell)})^{\frac{n-2\si }{2\si }}}
		+O(|x|)\notag\\
		=&:h^{(j)}(x).\label{uiu}
		\end{align}

		For sufficiently small $\delta>0,$  $u_{i}$ satisfy
		\begin{align}\label{po-ui}
		u_{i}(x)=c_{n,\si }c(n,\si )
		\int_{B_{\delta}(x_{i}^{(j)})}
		\frac{\wdt{K}(y)u_{i}(y)^{\frac{n+2\si }{n-2\si }}}{|x-y|^{n-2\si }}\,\ud y
		+h_{\delta}(x),
		\end{align}
		where
		\be\label{po-hdelta}
		h_{\delta}(x):=c_{n,\si }c(n,\si )
		\int_{\mathbb{R}^{n}\backslash B_{\delta}(x_{i}^{(j)})}
		\frac{\wdt{K}(y)u_{i}(y)^{\frac{n+2\si }{n-2\si }}}{|x-y|^{n-2\si }}\,\ud y.
		\ee
		By Proposition \ref{pro-pz}, we have
		\begin{align}
		&-\frac{2n}{n-2\si }
		\int_{B_{\delta}(x_{i}^{(j)})}(x-x_{i}^{(j)}) \cdot
		\nabla\wdt{K}(x) u_{i}(x)^{\frac{2n}{n-2\si }}\,\ud x\notag\\
		=&\frac{n-2\si }{2}\int_{B_{\delta}(x_{i}^{(j)})}
		\wdt{K}(x)u_{i}(x)^{\frac{n+2\si }{n-2\si }}h_{\delta}(x)\,\ud x\notag\\
		&\quad+\int_{B_{\delta}(x_{i}^{(j)})}(x-x_{i}^{(j)})\cdot\nabla h_{\delta}(x)
		\wdt{K}(x)u_{i}(x)^{\frac{n+2\si }{n-2\si }}\,\ud x\notag\\
		&\quad-\frac{2n}{n-2\si }\delta\int_{\partial B_{\delta}(x_{i}^{(j)})}
		\wdt{K}(x)u_{i}(x)^{\frac{2n}{n-2\si }}\,\ud s.	\label{pohoz}
		\end{align}
		
		Let $m_{ij}:=u_i(x_i^{(j)}),$
		by \cite[Lemma 2.18]{jlxm}, we have
		\begin{align*}
		|\nabla \wdt{K}(x_{i}^{(j)})|\leq Cm_{ij}^{-\frac{2(n-2\si -1)}{n-2\si }}.
		\end{align*}
		On the other hand, it follows from \eqref{kt} and \eqref{q1} that
		\begin{align*}
		\nabla\wdt{K}(x_{i}^{(j)})=
		\nabla Q_{(q^{(j)})}^{(n-2\si )}(x_i^{(j)})+
		o_{\delta}(1)|x_{i}^{(j)}|^{n-2\si -1},
		\end{align*}
		then, we  obtain
		\begin{align}\label{xij}
		|x_i^{(j)}|^{n-2\si -1}\leq C m_{ij}^{-\frac{2(n-2\si -1)}{n-2\si }}.
		\end{align}
		It follows from the above that
		\begin{align}
		&m_{ij}^{2}\Big|\int_{B_{\delta}}y\cdot \nabla R_{(q^{(j)})}(y+x_{i}^{(j)})
		u_{i}(y+x_{i}^{(j)})^{\frac{2n}{n-2\si }}\,\ud y\Big|\notag\\
		\leq &m_{ij}^{2}o_{\delta}(1)
		\int_{B_{\delta}}|y||y+x_{i}^{(j)}|^{n-2\si -1}u_i(y+x_i^{(j)})^{\frac{2n}{n-2\si }}\,\ud y\notag\\
		\leq& o_{\delta}(1)\int_{B_{\delta}}(|y|^{n-2\si }+|y||x_i^{(j)}|^{n-2\si -1})
		u_i(y+x_{i}^{(j)})^{\frac{2n}{n-2\si }}\,\ud y\notag\\
		=&o_{\delta}(1).\label{po-gk-re}
		\end{align}
		
		For  the left hand side of \eqref{pohoz},
		by Proposition \ref{pro-ibp-upb}, Proposition \ref{pro-esti}, \eqref{po-gk-re},
		and letting $i\to \infty$, we have
		\begin{align}
	&-\frac{2n}{n-2\si }
		m_{ij}^{2}\int_{B_{\delta}(x_{i}^{(j)})}(x-x_{i}^{(j)}) \cdot
		\nabla\wdt{K}(x) u_{i}(x)^{\frac{2n}{n-2\si }}\,\ud x\notag\\
		=&-\frac{2n}{n-2\si }m_{ij}^{2}
		\int_{B_{\delta}}y\cdot\nabla\wdt{K}(y+x_{i}^{(j)})u_{i}(y+x_{i}^{(j)})^{\frac{2n}{n-2\si }}\,\ud y\notag\\
		=&-\frac{2n}{n-2\si }m_{ij}^{2}\int_{B_{\delta}}y\cdot \nabla Q_{(q^{(j)})}^{(n-2\si )}(y+x_{i}^{(j)})\,\ud y+o_{\delta}(1)\notag\\
		=&-\frac{2n}{n-2\si }\int_{\mathbb{R}^{n}}
		\frac{z\cdot \nabla Q_{(q^{(j)})}^{(n-2\si )}(z+\xi^{(j)})}{(1+k^{(j)}|z|^{2})^{n}}\,\ud z
		+o_{\delta}(1),	\label{po-gk}
		\end{align}
		where $\xi^{(j)}=\lim_{i\to\infty}m_{ij}^{\frac{2}{n-2\si }}x_{i}^{(j)}$ and
		$k^{(j)}=K(q^{(j)})^{1/\si }/4$.
		
		For  the first term on the right hand side of \eqref{pohoz}, by Proposition \ref{pro-ibp-upb},
		Proposition \ref{pro-isbp-lim}, \eqref{uiu}, and letting $i\to \infty$, we have
		\begin{align}
		&m_{ij}^{2}\frac{n-2\si }{2}\int_{B_{\delta}(x_{i}^{(j)})}
		\wdt{K}(x)u_{i}(x)^{\frac{n+2\si }{n-2\si }}h_{\delta}(x)\,\ud x \notag\\
		=&m_{ij}^{2}\frac{n-2\si }{2}
		\int_{B_{\delta}(x_{i}^{(j)})}(\wdt{K}(x_{i}^{(j)})
		+(x-x_{i}^{(j)})\cdot \nabla\wdt{K}(x_{i}^{(j)})+O(|x-x_{i}^{(j)}|^{2}))
		h_{\delta}(x)
		u_{i}(x)^{\frac{n+2\si }{n-2\si }}\,\ud x\notag\\
		=&\frac{n-2\si }{2}
		\wdt{K}(x_{i}^{(j)})
		\int_{|y|\leq R_{i}}(m_{ij}^{-1} u_{i}(m_{ij}^{\frac{-2}{n-2\si }}y
		+x_{i}^{(j)}))^{\frac{n+2\si }{n-2\si }}h_{\delta}(m_{ij}^{-1}y+x_{i}^{(j)})\,\ud y+o(1)\notag\\
		=&\frac{n-2\si }{2}K(q^{(j)})\int_{\mathbb{R}^{n}}
		\frac{ b^{(j)}(0)}{(1+k^{(j)}|y|^{2})^{\frac{n+2\si }{2}}}\,\ud y.\label{po-bj0}
		\end{align}

		A direct calculation gives that when $|x-x_{i}^{(j)}|<\delta,$
		\begin{align}\label{gra-h-e}
		|\nabla h_{\delta}(x)|\leq
		\begin{cases}
		\displaystyle C\frac{|\delta^{2\si -1}-(\delta-|x-x_{i}^{(j)}|)^{2\si -1}|}
		{2\si -1}m_{ij}^{-1} & \text{ if }\, \si \ne{1}/{2}, \\
		\displaystyle C|\log \delta -\log(\delta-|x-x_{i}^{(j)}|)|m_{ij}^{-1} & \text{ if } \, \si ={1}/{2}.
		\end{cases}
		\end{align}
		For the second term on the right hand side of \eqref{pohoz},
		from \eqref{gra-h-e} and Proposition \ref{pro-ibp-upb}, we have
		\begin{align}
		&m_{ij}^{2}
		\Big|\int_{B_{\delta}(x_{i}^{(j)})}(x-x_{i}^{(j)})\cdot\nabla h_{\delta}(x)
		\wdt{K}(x)u_{i}(x)^{\frac{n+2\si }{n-2\si }}\,\ud x\Big|\notag\\
		\leq &Cm_{ij}
		\int_{B_{\delta}(x_{i}^{(j)})}|x-x_{i}^{(j)}|
		u_{i}(x)^{\frac{n+2\si }{n-2\si }}\,\ud x\notag\\
		=&Cm_{ij}^{1-\frac{2}{n-2\si }-\frac{2n}{n-2\si }+\frac{n+2\si
			}{n-2\si }}\int_{|y|<R_{i}}|y|(m_{ij}^{-1}
		u_{i}(m_{ij}^{-\frac{2}{n-2\si }}y+x_{i}^{(j)})^{\frac{n+2\si }{n-2\si }}\,\ud y\notag\\
		=&m_{ij}^{-\frac{2}{n-2\si }}\int_{\mathbb{R}^{n}}
		\frac{|y|}{(1+k^{(j)}|y|^{2})^{\frac{n+2\si }{2}}}\,\ud y=o(1).	\label{po-gh}
		\end{align}

		For the third term on the right side of \eqref{pohoz},
		 using Proposition \ref{pro-esti}, we have,
		\begin{align}
		&\lim_{i\to\infty}\Big|-m_{ij}^{2}\frac{2n}{n-2\si }\delta\int_{\partial B_{\delta}(x_{i}^{(j)})}
		\wdt{K}(x)u_{i}(x)^{\frac{2n}{n-2\si }}\,\ud s\Big|\notag\\
		\leq & C(\delta)\lim_{i\to \infty}
		m_{ij}^{2}m_{ij}^{-\frac{2n}{n-2\si }}=0.\label{po-par-b}
		\end{align}
		
		Let
		\begin{align}\label{lamb}
		\lambda_{j}:=K(q^{(j)})^{\frac{1-n+2\si }{2\si }}
		\lim_{i\to \infty }v_{i}(q^{(1)}_{i})v_{i}(q^{(j)}_{1}).
		\end{align}
		It follows from Propositions \ref{pro-ibp-lowb} and \ref{pro-isbp-upb}
		that $0<\lambda_{j}<\infty$.
		Therefore, by \eqref{pohoz}, \eqref{po-gk}, \eqref{po-bj0},
		\eqref{po-par-b}, \eqref{po-gh} and letting $\delta\to 0,$
		we have
		\begin{align}\label{Meq}
		-\frac{2n}{n-2\si }\int_{\mathbb{R}^{n}}
		\frac{z\cdot \nabla Q_{(q^{(j)})}^{(n-2\si )}(z+\xi^{(j)})}{(1+k^{(j)}|z|^{2})^{n}}\,\ud z
		=\frac{n-2\si }{2}K(q^{(j)})\int_{\mathbb{R}^{n}}
		\frac{ b^{(j)}(0)}{(1+k^{(j)}|y|^{2})^{\frac{n+2\si }{2}}}\,\ud y.
		\end{align}
		
		By \eqref{uiu-lim-a} and \eqref{lamb}, we have
		\begin{align}\label{uiu-lim-b}
		b^{(j)}(0)=2^{\frac{3(n-2\si )}{2}}\sum_{\ell\ne j}\frac{\lambda_{\ell}}{\lambda_{j}}
		\frac{K(q^{(j)})^{\frac{1-n+2\si }{2\si }}}{K(q^{(\ell)})^{\frac{1}{2\si }}}
		G_{q^{(\ell)}}(q^{(j)}).
		\end{align}
		Substituting \eqref{uiu-lim-b} into \eqref{Meq} and
		making a change of variable, we have
		\begin{align*}
		&-\frac{2n}{n-2\si }2^{2(n-\si )}K(q^{(j)})^{\frac{\si -n}{\si }}
		\int_{\mathbb{R}^{n}}\frac{y\cdot
			\nabla Q_{q^{(j)}}^{(n-2\si )}(y+\sqrt{k^{j}}\xi^{(j)})}{(1+|y|^{2})^{n}}\,\ud y\\
		=&(n-2\si )2^{n-2+\frac{3(n-2\si )}{2}}\frac{2\pi^{n/2}}{\Gamma(\si +n/2)}
		\sum_{\ell\ne j}\frac{\lambda_{\ell}}{\lambda_{j}}
		\frac{K(q^{(j)})^{\frac{1-2n+4\si }{2\si }}}{K(q^{(\ell)})^{\frac{1}{2\si }}}
		G_{q^{(\ell)}}(q^{(j)}),
		\end{align*}
		where $\eta^{(j)}:=\sqrt{k^{(j)}}\xi^{(j)}.$ It follows that
		\begin{align}
		&-K(q^{(j)})^{-\frac{1+\si }{\si }}\lambda_{j}
		\int_{\mathbb{R}^{n}}\frac{y\cdot
			\nabla Q_{q^{(j)}}^{(n-2\si )}(y+\eta^{(j)})}{(1+|y|^{2})^{n}}\,\ud y\notag\\
		=&\frac{2^{\frac{n-2\si }{2}}(n-2\si )^2}{4n}
		\frac{\pi^{n/2}}{\Gamma(\si +\frac{n}{2})}
		\sum_{\ell\ne j}
		\frac{G_{q^{(\ell)}}(q^{(j)})}{K(q^{(\ell)})^{\frac{1}{2\si }}
			K(q^{(j)})^{\frac{1}{2\si }}}
		\lambda_{\ell}.\label{M-lambda}
		\end{align}
		We next prove that $q^{(j)}$, $\eta^{(j)}$ ($j=1,\ldots,k$),
		satisfy \eqref{gkn-2s} with $q_0=q^{(j)}$, $\eta_0=\eta^{(j)}$.
		In fact, due to \eqref{M-lambda}, we only need to prove
		\begin{align}\label{eq-eta-q}
		\int_{\mathbb{R}^{n}} \nabla Q_{(q^{(j)})}^{(n-2\si )}(y+\eta^{(j)})(1+|y|^{2})^{-n} \,\ud y=0.
		\end{align}

		We first claim that
		\begin{align}\label{gra-k-ui}
		\int_{B_{\delta}}\nabla \wdt{K}(x+x_{i}^{(j)})
		u_{i}(x+x_{i}^{(j)})^{\frac{2n}{n-2\si }}\,\ud x=O(m_{ij}^{-2}).
		\end{align}
		Indeed, by using symmetry, we have
		\begin{align*}
		&\frac{n-2\si }{2n}\int_{B_{\delta}(x_i^{(j)})}\wdt{K}(x)\nabla u_{i}(x)^{\frac{2n}{n-2\si }}\,
		\ud x\\
		=&\int_{B_{\delta}(x_{i}^{(j)})} \wdt{K}(x)u_{i}(x)^{\frac{n+2\si }{n-2\si }}
		\nabla u_{i}(x)\,
		\ud x\\
		=&(2\si -n)\int_{B_{\delta}(x_{i}^{(j)})} \wdt{K}(x)u_{i}(x)^{\frac{n+2\si }{n-2\si }}
		\int_{B_{\delta}(x_{i}^{(j)})}
		\frac{(x-y)\wdt{K}(y)u_i(y)^{\frac{n+2\si }{n-2\si }}}{|x-y|^{n-2\si +2}}\,\ud y\, \ud x\\
		&\quad
		+\int_{B_{\delta}(x_{i}^{(j)})}\wdt{K}(x)u_i(x)^{\frac{n+2\si }{n-2\si }}
		\nabla h_{\delta}(x)\,\ud x\\
		=&\int_{B_{\delta}(x_{i}^{(j)})}\wdt{K}(x)u_i(x)^{\frac{n+2\si }{n-2\si }}
		\nabla h_{\delta}(x)\,\ud x,
		\end{align*}
		and by the divergence theorem,
		\begin{align*}
		&\int_{B_{\delta}(x_i^{(j)})}\wdt{K}(x)
		\nabla u_i(x)^{\frac{2n}{n-2\si }}\,\ud x\\
		=&-\int_{B_{\delta}(x_i^{(j)})}\nabla\wdt{K}(x) u_{i}(x)^{\frac{2n}{n-2\si }}\,\ud x
		+\delta\int_{\partial B_{\delta}(x_i^{(j)})}
		\wdt{K}(x)u_{i}(x)^{\frac{2n}{n-2\si }}(x-x_{i}^{(j)})\,\ud s.
		\end{align*}
		It follows that
		\begin{align*}
		&-\frac{n-2\si }{2n}
		\int_{B_{\delta}(x_i^{(j)})}\nabla\wdt{K}(x) u_{i}(x)^{\frac{2n}{n-2\si }}\,\ud x\\
		&\quad+\frac{n-2\si }{2n}\delta\int_{\partial B_{\delta}(x_i^{(j)})}
		\wdt{K}(x)u_{i}(x)^{\frac{2n}{n-2\si }}(x-x_{i}^{(j)})\,\ud s\\
		=&\int_{B_{\delta}(x_{i}^{(j)})}\wdt{K}(x)u_i(x)^{\frac{n+2\si }{n-2\si }}
		\nabla h_{\delta}(x)\,\ud x.
		\end{align*}
		Then by using Proposition \ref{pro-isbp-upb} and \eqref{gra-h-e}, we have
		\begin{align*}
		&\int_{B_{\delta}(x_i^{(j)})}\nabla\wdt{K}(x) u_{i}(x)^{\frac{2n}{n-2\si }}\,\ud x\\
		=&\delta\int_{\partial B_{\delta}(x_i^{(j)})}
		\wdt{K}(x)u_{i}(x)^{\frac{2n}{n-2\si }}(x-x_{i}^{(j)})\,\ud s-\frac{2n}{n-2\si }\int_{B_{\delta}(x_{i}^{(j)})}\wdt{K}(x)u_i(x)^{\frac{n+2\si }{n-2\si }}
		\nabla h_{\delta}(x)\,\ud x\\
		=&O(m_{ij}^{-2}).
		\end{align*}
		Therefore, \eqref{gra-k-ui} can be obtained from the above.

		Multiplying \eqref{gra-k-ui} by $m_{ij}^{\frac{2}{n-2\si }(n-2\si -1)}$,
		we have
		\begin{align}\label{pro-eq3}
		&\int_{B_{\delta}}\nabla Q_{(q^{(j)})}^{(n-2\si )}(m_{ij}^{\frac{2}{n-2\si }}
		x+m_{ij}^{\frac{2}{n-2\si }}x_{i}^{(j)})u_{i}(x+x_{i}^{(j)})^{\frac{2n}{n-2\si }}\,\ud x\notag\\
		=&o_{\delta}(1)\int_{B_{\delta}}|m_{ij}^{\frac{2}{n-2\si }}
		x+m_{ij}^{\frac{2}{n-2\si }}x_{i}^{(j)}|^{n-2\si -1}
		u_{i}(x+x_{i}^{(j)})^{\frac{2n}{n-2\si }}\,\ud x+o(1).
		\end{align}
		By using \eqref{xij} and Proposition \ref{pro-esti}, we have
		\begin{align}
		&m_{ij}^{\frac{2}{n-2\si }(n-2\si -1)}\Big|\int_{B_{\delta}}
		\nabla R_{(q^{(j)})}(x+x_{i}^{(j)}) u_{i}(x+x_i^{(j)})^{\frac{2n}{n-2\si }}\,\ud x\Big|\notag\\
		\leq &m_{ij}^{\frac{2}{n-2\si }(n-2\si -1)}\int_{B_{\delta}}
		|x+x_i^{(j)}|^{n-2\si -1} u_{i}(x+x_i^{(j)})^{\frac{2n}{n-2\si }}\,\ud x\notag\\
		\leq &
		m_{ij}^{\frac{2}{n-2\si }(n-2\si -1)}\int_{B_{\delta}}|x|^{n-2\si -1}
		u_{i}(x+x_i^{(j)})^{\frac{2n}{n-2\si }}\,\ud x\notag\\
		&+
		m_{ij}^{\frac{2}{n-2\si }(n-2\si -1)}\int_{B_{\delta}}|x_i^{(j)}|^{n-2\si -1}
		u_{i}(x+x_i^{(j)})^{\frac{2n}{n-2\si }}\,\ud x \notag\\
		=& O(1),\label{pro-eq1}
		\end{align}
		and using Proposition \ref{pro-ibp-upb},
		\begin{align}
		&\int_{B_{\delta}}\nabla Q_{(q^{(j)})}^{(n-2\si )}(m_{ij}^{\frac{2}{n-2\si }}
		x+m_{ij}^{\frac{2}{n-2\si }}x_{i}^{(j)})u_{i}(x+x_{i}^{(j)})^{\frac{2n}{n-2\si }}\,\ud x\notag\\
		=&\int_{|x|\leq r_{i}}
		\nabla Q_{(q^{(j)})}^{(n-2\si )}(m_{ij}^{\frac{2}{n-2\si }}
		x+m_{ij}^{\frac{2}{n-2\si }}x_{i}^{(j)})u_{i}(x+x_{i}^{(j)})^{\frac{2n}{n-2\si }}\,\ud x+o(1)\notag\\
		=&\int_{|y|\leq R_{i}}
		Q_{(q^{(j)})}^{(n-2\si )}(y+m_{ij}^{\frac{2}{n-2\si }}x_{i}^{(j)})
		(m_{ij}^{-1}u_{i}(m_{ij}^{\frac{2}{n-2\si }}y+x_{i}^{(j)}))^{\frac{2n}{n-2\si }}\,\ud y+o(1).\label{pro-eq2}
		\end{align}
		By \eqref{pro-eq3}--\eqref{pro-eq2},
		and letting $i\to \infty$ and $\delta\to 0$, we have
		\begin{align*}
		\int_{\mathbb{R}^{n}}\nabla Q_{(q^{(j)})}^{(n-2\si )}(z+\xi^{(j)})(1+k^{(j)}|z|^{2})^{-n}\,\ud z
		=0.
		\end{align*}
		Making a change of variable, we  establish \eqref{eq-eta-q}.
		Theorem \ref{pro-lambda-exi} follows from \eqref{M-lambda} and \eqref{eq-eta-q}.
	\end{proof}

	\section{Perturbation method  for existence results}\label{sec:3}
	In this section, we use  the method in  \cite{JLXOn2015} to
	establish a perturbation result  Theorem \ref{prop-theorem3.1}.
	Similar results in the classical Nirenberg problem were obtained in \cite{perturbation,ChangYangPrescribing1987,LPrescribing1995}. In this section, we only concern with the case $1\leq \si<n/2$, since the proof of $0<\si<1$ can be found in \cite[Section 3]{JLXOn2015}.

	For a  conformal transformation $\varphi_{P, t}$  (see  \eqref{eq:varphiPt2}),
	we let
$$
	T_{\varphi_{P, t}} v=v \circ \varphi_{P, t}|\operatorname{det} \mathrm{d} \varphi_{P, t}|^{
\frac{n-2\si}{2 n}},
	$$
where $ \mathrm{d} \varphi_{P, t}$ denotes the Jacobian of $\varphi_{P, t}$ satisfying
$$
\varphi_{P,t}^{*} g_{0}=|\det \ud\varphi_{P,t}|^{2/n} g_{0}.
$$
	Let
	\begin{align*}
	\mathscr{S}&:=\Big\{v \in H^{\si}(\Sn):\dashint_{\Sn} |v|^{\frac{2 n}{n-2 \si}} \,\mathrm{d}  \mathrm{vol}_{g_{0}}=1\Big\},  \\
	\mathscr{S}_{0}&:=\Big\{v \in \mathscr{S}:\dashint_{\Sn}  x|v|^{\frac{2 n}{n-2 \si}} \,\mathrm{d}  \mathrm{vol}_{g_{0}}=0\Big\}.
	\end{align*}
	For $w \in \mathscr{S}_{0}$ and $p\in  B^{n+1}$,  let  $\pi (w,p)$ be defined by $\pi(w,0)=w$  and  $\pi(w, p)=T_{\varphi_{P, t}}^{-1} w$.  It can be checked that the map $\pi: \mathscr{S}_{0} \times B^{n+1} \rightarrow \mathscr{S}$ is a $C^{2}$ diffeomorphism, see \cite{LPrescribing1995}.
	
	
	It is easy to see that $1\in \mathscr{S} \cap \mathscr{S}_{0}$.  By some direct computations and  elementary properties of spherical harmonics,  we have
	$$
	T_{1} \mathscr{S}=\Big\{\phi: \int_{\Sn} \phi=0\Big\}=\operatorname{span}\{\text {spherical harmonics of degree} \geq 1\},
	$$
	and  	$$
	T_{1} \mathscr{S}_{0}=\operatorname{span}\{\text {spherical harmonics of degree} \geq 2\},
	$$
	where 	$T_{1} \mathscr{S}$ and 	$T_{1} \mathscr{S}_{0}$ denote the tangent space of at $1$, respectively.
	
	Let us use $\widetilde{w} \in T_{1} \mathscr{S}_{0}$ as local
	coordinates of $w \in \mathscr{S}_{0}$ near $w=1$,
	and $\widetilde{w}=0$ corresponds to $w=1$. Using implicit function theorem, it is easy to check that $\mathscr{S}_{0}$ is represented locally near $1$ as a graph over $T_{1} \mathscr{S}_{0}$:  For all $\widetilde{w} \in T_{1} \mathscr{S}_{0}$, $\widetilde{w}$ sufficiently  close to 0, there is a twice differentiable map $\mu(\widetilde{w}) \in \mathbb{R}$, $\eta(\widetilde{w}) \in \mathbb{R}^{n+1}$ defined in a neighborhood of 0 in $T_{1} \mathscr{S}_{0}$ with $\mu(0)=0$, $\eta(0)=0$, $D \mu(0)=0$, and $D \eta(0)=0$
	such that 
	\begin{align*}
	\dashint_{\Sn} |1+\widetilde{w}+\mu+\eta \cdot x|^{2 n /(n-2 \si)}=1
	\end{align*}
	and
	\begin{align*}
	\dashint_{\Sn}|1+\widetilde{w}+\mu+\eta \cdot x|^{2 n /(n-2 \si)} x=0.
	\end{align*}
	
	

	Now we consider a functional on $\mathscr{S}$:
	\begin{align}\label{EK}
	E_{K}(v)=\frac{\dashint_{\Sn } v P_{\si }(v) \,\mathrm{d} \mathrm{v o l}_{g_{0}}}
	{(\dashint_{\Sn } K|v|^{2 n /(n-2 \si )}\,
		\mathrm{d} \mathrm{v o l}_{g_{0}})^{(n-2 \si ) / n}}.
	\end{align}
	We have
	\begin{proposition}\label{lem3.6}
		Let $n\geq 3$, $1\leq \si<n/2$, and  $K\in C^{1}(\Sn )$ be a positive function.
		There exist some constants $\varepsilon_{1}=\varepsilon_{1}(n, \si )>0$ and
		$\varepsilon_{2}=\varepsilon_{2}(n, \si )>0$, such that,
		if $\|K-1\|_{L^{\infty}(\Sn )}  \leq \varepsilon_{1}$, then
		$$ \min _{w \in \mathscr{S}_{0},\,\|w-1\|_{H^{\si}(\Sn)} \leq \varepsilon_{2}} E_{K}(w)$$
		has a unique minimizer $w_{K}>0$ .
		Furthermore, $D^{2} E_{K}|_{\mathscr{S}_{0}}(w_{K})$ is positive definite,
		and there exists a constant $C=C(n,\si)$  such that
		\begin{align}\label{wk-1}
		\|w_{K}-1\|_{H^{\si}(\Sn)} \leq C
		\inf _{c \in \mathbb{R}}\|K-c\|_{L^{2 n /(n+2 \si )}(\mathbb{S}^n)}.
		\end{align}
	\end{proposition}
	
	\begin{proof}
		
		Using the conformal invariance of $P_{\si}$ and \eqref{EK}, for $\widetilde{w} \in T_{1} \mathscr{S}_{0}$ and $\widetilde{w}$   close to 0, we have
		$$
		\widetilde{E}(\widetilde{w}):=E_{1}(w)=\dashint_{\Sn}  w P_{\si }(w),
		$$
	where $w=1+\widetilde{w}+\mu(\widetilde{w})+\eta(\widetilde{w}) \cdot x$.

		It is well known (see \cite{mor})
		that $P_{\si }$ has eigenfunctions the spherical harmonics and eigenvalues
		\begin{align}\label{psi-eig}
		\lambda_{k}=\frac{\Gamma(k+\frac{n}{2}+\si )}{\Gamma(k+\frac{n}{2}-\si )}, \quad k \geq 0,
		\end{align}
		with multiplicity $(2 k+n-1)(k+n-2) ! /(n-1) ! k ! $.
		It follows that
		\begin{align}\label{ew1}
		\widetilde{E}(\widetilde{w})=P_{\si }(1)(1+2 \mu(\widetilde{w}))+\int_{\Sn } \widetilde{w} P_{\si }(\widetilde{w})+O(\|\widetilde{w}\|_{H^{\si }(\Sn )}^{2}).
		\end{align}
		Thus,
		\begin{align}\label{muw}
		\mu(\widetilde{w})
		&=-\frac{1}{2} \cdot \frac{n+2 \si }{n-2 \si } \int_{\Sn } \widetilde{w}^{2}+o(\|\widetilde{w}\|_{H^{\si }(\Sn )}^{2}).
		\end{align}
		Note that $\lambda_{1}=\frac{n+2 \si}{n-2 \si} P_{\si}(1)$ (see \eqref{psi-eig}), then,	by \eqref{ew1} and \eqref{muw},
		\begin{align}\label{3.7-1}
		\widetilde{E}(\widetilde{w})=P_{\si}(1)+\dashint_{\Sn}(\widetilde{w} P_{\si}(\widetilde{w})-\lambda_{1} \widetilde{w}^{2})+o(\|\widetilde{w}\|_{H^{\si }(\Sn )}^{2}).
		\end{align}
		Set $Q(\widetilde{w}):=\dashint_{\Sn}(\widetilde{w} P_{\si}(\widetilde{w})-\lambda_{1} \widetilde{w}^{2})$.
		It is clear  that for any $\widetilde{w}$,	$\widetilde{v} \in T_{1} \mathscr{S}_{0}$,
		\begin{align}\label{3.7-2}
		D^{2} Q(\widetilde{w})(\widetilde{v}, \widetilde{v}) =2\dashint_{\Sn} (\widetilde{v} P_{\si}(\widetilde{v})-\lambda_{1} \widetilde{v}^2)
		\geq 2\Big(1-\frac{\lambda_{1}}{\lambda_{2}}\Big)\|\widetilde{v}\|_{H^{\si }(\Sn )}^{2},
		\end{align}
		which means the quadratic form $Q(\widetilde{w})$ is positive definite in $T_{1} \mathscr{S}_{0}$. Furthermore,  we derive from Sobolev inequality \eqref{sobolev} that for some $\varepsilon_{0}=\varepsilon_{0}(n, \si)>0$,
		\begin{align}\label{3.7-3}
		\|E_{K}|_{\mathscr{S}_{0}}-E_{1}|_{\mathscr{S}_{0}}\|_{C^{2}(B_{\varepsilon_{0}}(1))} \leq O(\varepsilon),
		\end{align}
		provided $\|K-1\|_{L^{\infty}(\Sn)} \leq \varepsilon$. Here $B_{\varepsilon_{0}}(1)$ denotes the ball in $\mathscr{S}_{0}$ of radius $\varepsilon_{1}$ centered at 1. Then we verify by direct computations that,  for any $\widetilde{w} \in T_{1} \mathscr{S}_{0}$ and   any constant $c$,
		$$
		\langle D E_{K}|_{\mathscr{S}_{0}}(1), \widetilde{w}\rangle=-2 P_{\si}(1)
		\Big(\dashint_{\Sn} K\Big)^{(2 \si-2 n) / n}\dashint_{\Sn}(K-c) \widetilde{w}.
		$$
		Therefore,
		\begin{align}\label{3.7-4}
		\|D E_{K}|_{\mathscr{S}_{0}}(1)\| \leq C\|K-c\|_{L^{2 n /(n+2 \si)}(\Sn)}.
		\end{align}
		As a consequence, we see from \eqref{3.7-1}--\eqref{3.7-3}
		that the minimizing problem has a unique minimizer $w_{K}$ and
		$D^{2}E_{K}|_{\mathscr{S}_{0}}(w_{K})$ is positive definite. Estimate
		\eqref{wk-1} follows from \eqref{3.7-2}--\eqref{3.7-4} with  some standard functional analysis arguments.
		
		We are left to prove  the positivity of $w_K$.
		
		Since $w_{K}$ is a constrained local minimum,
		$w_{K}$ satisfies the Euler-Lagrange equation for some Lagrange multiplier $\Lambda_{K} \in \mathbb{R}^{n+1}$ :
		\begin{align}\label{lem-wk}
		P_{\si }(w_{K})=(\lambda_{K} K-\Lambda_{K} \cdot x)|w_{K}|^{4 \si  /(n-2 \si )} w_{K} \quad \text { on }\, \Sn ,
		\end{align}
		where
		$$
		\lambda_{K}=\frac{\dashint_{\Sn} w_{K} P_{\si }(w_{K})}{\dashint_{\Sn } K|w_{K}|^{\frac{2 n}{n-2 \si }}} .
		$$
		Then,	using the same 	argument in  \cite[Lemma 3.6]{JLXOn2015}, we   obtain  $w_{K}\geq 0$.	Note that equation \eqref{lem-wk} can be  equivalently rewritten as
		\begin{align*}
		w_{K}(\xi)=c_{n,\si}\int_{\Sn} \frac{(\lambda_{K} K(\eta)-\Lambda_{K} \cdot \eta)w_{K}(\eta)^{(n+2 \si) /(n-2 \si)}}{|\xi-\eta|^{n-2 \si}} \,\mathrm{d} \eta,
		\end{align*}
		by using \eqref{gr}.	If there exists some $\xi_{0}\in \Sn$ such that $w_{K}(\xi_{0})=0$,
		then using the facts
		$\|K-1\|_{L^{\infty}(\Sn)}\leq \varepsilon$, $|\lambda_{k}-c(n,\si)|=O(\varepsilon)$, and $|\Lambda_{K}|=O(\varepsilon)$, we immediately obtain $(\lambda_{K} K-\Lambda_{K} \cdot x)>0$
		for sufficient small $\va$.
		It follows that $w_{K}\equiv0$. However, it is a contradiction because $w_K\in \mathscr{S}_0$, which in turn
		implies  that $w_{K}>0$.  The proof is finished.
	\end{proof}
	As  illustrated before, we  write  $v=\pi(w, p)=T_{\varphi_{P, t}}^{-1} w$ for any $v \in \mathscr{S}$ with  $w \in \mathscr{S}_{0}$, $p=s P\in B^{n+1}$, $t\geq 1$ and $s=(t-1) / t$.  It is easy to check that  $E_{K}(v)=E_{K \circ \varphi_{P, t}}(w)$. Let us rewrite $E_{K}(v)$ in the $(w, p)$ variables:
	\begin{align}\label{Ip}
	I(w, p):=E_{K}(v)=E_{K \circ \varphi_{P, t}}(w).
	\end{align}
	By Proposition \ref{lem3.6}, for  some $\va_{2}>0$, we have
	\begin{align}\label{3.16}
	\min _{w \in \mathscr{S}_{0},\|w-1\|_{H^{\si}(\Sn)} \leq \varepsilon_{2}} I(w, p)=\min _{w \in \mathscr{S}_{0},\|w-1\|_{H^{\si}(\Sn)} \leq \varepsilon_{2}} E_{K \circ \varphi_{P, t}}(w).
	\end{align}
	Furthermore, if $\|K-1\|_{L^{\infty}(\Sn )} \leq \va_1$,	the minimizer exists and we denote it as $w_{p}$.
	
	Let 	$\va_{2}$ be as in Proposition \ref{lem3.6},	we define
	\begin{align}\label{nn1}
	\mathcal{N}_{1}:=\{w \in \mathscr{S}_{0} :\|w-1\|_{H^{\si }(\Sn )} \leq \varepsilon_{2}\}.
	\end{align}
	For $1<t\leq \infty$, define
	\begin{align*}
	&\mathcal{N}_{2}(t):=\Big\{v \in \mathscr{S}:
	v=\pi (w, p)\, \text { for some }\, w \in \mathcal{N}_{1}\, \text { and }\, p=s P,   \\
	&\quad\quad\quad\quad\quad  P \in \Sn , s=\frac{\zeta-1}{\zeta}, 1 \leq \zeta<{t}\Big\},
	\end{align*}
	and
	\begin{align*}
	\mathcal{N}_{3}(t):=\{v\in H^{\si }(\Sn )\backslash\{0\}:
	cv\in \mathcal{N}_2(t) \text{ for some constant } c>0\}.
	\end{align*}

Using the Proposition \ref{lem3.6} and the properties of the integral equation,
as well as a natural fibration of $H^{\sigma}$, we can obtain the following perturbation result:
	\begin{theorem}\label{prop-theorem3.1}
		Let $n\geq 2$, $0<\si<n/2$, and  $K \in C^{1}(\Sn )$ ($K\in C^{1,1}(\Sn )$
		if $\si \leq 1/2$) be a positive nonconstant function
		and  $\varphi_{P, t}$ be as in \eqref{eq:varphiPt2}  for $P\in \Sn $, $1\leq t<\infty$.
		Suppose that there exists some constant
		$\varepsilon_{3}=\varepsilon_{3}(n) \in(0, \varepsilon_{1})$
		such that
		$\|K-1\|_{L^{\infty}(\Sn )}\leq \va_{3}$.
		Suppose also that for all $P\in \Sn $, we have
		$$\|K \circ \varphi_{P, t}-K(P)\|_{L^{2}(\Sn )}^{2} \leq o\Big(\Big|\int_{\Sn } K \circ \varphi_{P, t}(x) x\Big|\Big)
		\quad \text{ as }\, t\to \infty,$$
		and
		$$
		\operatorname{deg}\Big(\int_{\Sn } K \circ \varphi_{P, t}(x) x, B^{n+1}, 0\Big) \neq 0
		\quad\text{ for large }\, t.
		$$
		Then \eqref{main-eq} has at least one positive solution.
		Furthermore, for any $\alpha \in(0,1)$ satisfying that $\alpha+2 \si $
		is not an integer, there exists  constant $C_{1}>0$ depending only on $n, \alpha, \si ,$
		such that for all $C \geq C_{1}$,
		\begin{align}
		&\mathrm{deg}\Big(v-(P_{\si })^{-1} K|v|^{4 \si  /(n-2 \si )} v, \mathcal{N}_{3}(t) \cap\{v \in C^{2 \si +\alpha}:\|v\|_{C^{2 \si +\alpha}}<C\}, 0\Big) \notag\\
		&\quad=(-1)^{n} \mathrm{deg}\Big(\int_{\Sn } K \circ \varphi_{P, t}(x) x, B^{n+1}, 0\Big).\label{prop-1}
		\end{align}
	\end{theorem}
	
	\begin{proof}
		We initiate the proof with Proposition \ref{lem3.6} since the case $\si\in(0,1/2)$ follows from a degree argument, see \cite{JLXOn2015}. For each $p\in B_1$, let $w_p$ is the minimizer of \eqref{3.16},  set
		\begin{align*}
		\mathscr{A}_{p}=\frac{1}{n}\dashint_{\Sn }\langle\nabla(K \circ \varphi_{P, t}), \nabla x\rangle w_{p}^{2 n /(n-2 \si)},\quad  \mathscr{B}_{p}=\dashint_{\Sn } K \circ \varphi_{P, t}(x) x.
		\end{align*}
		It is clear that $\mathscr{B}_p \neq 0$. We write
		$$
		\mathscr{A}_p=\mathscr{B}_{p}+\mathrm{I}+\mathrm{II},
		$$
		where
		\begin{align*}
		\mathrm{I}=& \dashint_{\Sn }(K \circ \varphi_{P, t}-K(P)) x(w_{p}^{2 n /(n-2 \si )}-1), \\
		\mathrm{II}=&-\frac{1}{n} \dashint_{\Sn }
		(K \circ \varphi_{P, t}-K(P))\langle\nabla x,
		\nabla(w_{p}^{2 n /(n-2 \si )})\rangle.
		\end{align*}
		The proof consists of 3 steps.
		
		\textbf{Step 1:} Estimates of $\mathrm{I}$.

		By using  Cauchy-Schwartz inequality and \eqref{wk-1},  we have, as $t\to\infty$,
		\begin{align}
		|\mathrm{I}| & \leq C\|K \circ \varphi_{P, t}-K(P)\|_{L^{2}(\Sn )}\|w_{p}^{2 n /(n-2 \si )}-1\|_{L^{2}(\Sn )}\notag \\
		& \leq C\|K \circ \varphi_{P, t}-K(P)\|_{L^{2}(\Sn )}
		\|w_{p}-1\|_{H^{\si}(\Sn)}\notag\\
		& \leq C\|K \circ \varphi_{P, t}-K(P)\|_{L^{2}(\Sn )}^2\notag\\
		&\leq  o(t) \mathscr{B}_{p}  .\label{pro-eq4}
		\end{align}
		
		\textbf{Step 2:} Estimates of $\mathrm{II}$.
		
		Firstly, let us  claim
		that  there exists a constant $C$ depending on $n,\si ,\va_{1}$ such that
		\begin{align}\label{estimates2}
		\|\nabla(w_{p}^{2 n /(n-2 \si )})\|_{L^{2}(\Sn )}
		\leq C\|K\circ \varphi_{P, t}-K(P)\|_{L^2(\Sn )},
		\end{align}	
		here $\va_{1}$ is given by Proposition \ref{lem3.6}. Once we verify it, together with  Cauchy-Schwartz inequality, it will yield that,  as $t\to\infty$,
		\begin{align}
		|\mathrm{II}| \leq C\|K \circ \varphi_{P, t}-K(P)\|_{L^{2}(\Sn )} \|\nabla(w_{p}^{2 n /(n-2 \si )})\|_{L^{2}(\Sn )}
		\leq o(t) \mathscr{B}_{p}.\label{pro-eq5}
		\end{align}
		Now we give the proof of the claim.

		As in \eqref{lem-wk}, we  know that
		\begin{align}\label{eq-wp}
		P_{\si}(w_{p})=(\lambda_p K \circ \varphi_{P, t}-
		\Lambda_{p} \cdot x) w_{p}^{(n+2 \si) /(n-2 \si)} \quad \text { on }\, \Sn ,
		\end{align}
		where
		\begin{align}\label{lamp}
		\lambda_{p}=\frac{\dashint_{\Sn } w_{p} P_{\si }(w_{p})}{\dashint_{\Sn } K \circ \varphi_{P, t} w_{p}^{\frac{2 n}{n-2 \si }}}
		\end{align}
		and $\Lambda_{p}\in \mathbb{R}^{n+1}$. It is easy to see $w_{p}$ solves \eqref{main-eq} if and only if $\Lambda_{p}=0$.
		
		Denote $v_{p}=w_{p}-1$,  using Taylor's theorem and \eqref{eq-wp}, we get
		\begin{align}
		P_{\si}(v_p)=&(\lambda_p K \circ \varphi_{P, t}-
		\Lambda_{p} \cdot x) w_{p}^{(n+2 \si ) /(n-2 \si )}-P_{\si}(1) \notag
		\\
		=&(\lambda_{p}-P_{\si}(1))K \circ \varphi_{P, t}+P_{\si }(1)(K\circ\varphi_{P, t}-1)
		-\Lambda_{p}\cdot x \notag \\
		&+\frac{n+2\si}{n-2\si}(\lambda_p K \circ \varphi_{P, t}-
		\Lambda_{p} \cdot x) v_p \notag  \\
		&+(\lambda_p K \circ \varphi_{P, t}-
		\Lambda_{p} \cdot x)
		o(|v_p|)\notag\\=&	\mathcal{R}(x)+\frac{n+2\si}{n-2\si}\mathcal{Q}(x)v_p+o(|v_p|),\label{eq-wp-2}
		\end{align}
		where
		\begin{align*}
		\mathcal{R}(x)=&(\lambda_{p}-P_{\si}(1))K \circ \varphi_{P, t}+P_{\si }(1)(K\circ\varphi_{P, t}-1)
		-\Lambda_{p}\cdot x,\\
		\mathcal{Q}(x)=&(\lambda_p K \circ \varphi_{P, t}-\Lambda_{p}\cdot x).
		\end{align*}

		By the stereographic projection and Green's representation \eqref{gr},	 we have (up to a  harmless constant)
		\begin{align*}
		\widetilde{v}_p(y)=\int_{\rn}\frac{|J_{F}|^{\frac{2\si}{n}}
			\mathcal{Q}(F(y))
			(\frac{n+2\si}{n-2\si}\widetilde{v}_p+o(|\widetilde{v}_p|))
			+|J_{F}|^{\frac{n-2\si}{2n}} \mathcal{R}(F(y)) }{|z-y|^{n-2\si}}\,\ud z,
		\end{align*}
		where  $\widetilde{v}_p=|J_{F}|^{2n/(n-2\si)}(v_p \circ F)$ as illustrated in the introduction.		We consider
		\begin{align*}
		\widetilde{v}_p(y)=&\int_{B_{3}}
		\frac{\frac{n+2\si}{n-2\si}|J_{F}|^{\frac{2\si}{n}}\mathcal{Q}(F(z))\widetilde{v}_p(y)}{|y-z|^{n-2\si}}\\
		&\quad+\Big(\int_{\rn\backslash B_{3}}
		\frac{\frac{n+2\si}{n-2\si}|J_{F}|^{\frac{2\si}{n}}\mathcal{Q}(F(z))\widetilde{v}_p(y)}{|y-z|^{n-2\si}}
		+\int_{\rn} \frac{|J_{F}|^{\frac{2\si}{n}}\mathcal{Q}(F(z))o(|\widetilde{v}_p|)
			+|J_{F}|^{\frac{n+2\si}{2n}}\mathcal{R}(F(z))}{|z-y|^{n-2\si}}\Big)\\
		=:&\int_{B_{3}}
		\frac{\frac{n+2\si}{n-2\si}|J_{F}|^{\frac{2\si}{n}}\mathcal{Q}(F(z))\widetilde{v}_p(y)}{|y-z|^{n-2\si}}
		+\mathcal{H}(x).
		\end{align*}
		Now we give an upper bound of  $\|\mathcal{H}\|_{L^{\infty}(\rn)}$.

		It follows from Proposition \ref{lem3.6} that
		\begin{align}\label{lamp-p1}
		|\lambda_{p}-P_{\si }(1)|
		=\Big|\frac{\|v_p\|_{H^{\si}(\Sn ) }+2P_{\si}(1)\dashint_{\Sn }v_p  }{\dashint_{\Sn }K\circ \varphi_{P, t}|w_{p}|^{\frac{2n}{n-2\si}} }\Big|
		\leq O (\|K\circ\varphi_{P, t}-1\|_{L^{\infty}(\Sn )}).
		\end{align}
		Multiplying	\eqref{eq-wp} by $\Lambda_{p} \cdot x$ and integrating over both sides we have
		\begin{align}\label{Lamp}
		-\int_{\Sn }
		(\Lambda_{p} \cdot x)^{2}w_{p}^{\frac{n+2\si }{ n-2 \si }}
		=&
		\lambda_{1}\int_{\Sn } w_{p} \Lambda_{p} \cdot x
		-\lambda_{p} \int_{\Sn } K\circ\varphi w_{p}^{\frac{n+2 \si }{ n-2 \si }}
		\Lambda_{p} \cdot x.
		\end{align}
		It is easy to see that
		\begin{align}\label{Lamp-p1}
		|\Lambda_{p}|=O(\|K\circ\varphi_{P, t}-1\|_{L^{\infty}(\Sn )}).
		\end{align}
		Meanwhile, we get from  \eqref{lamp-p1} and \eqref{Lamp-p1} that
		\begin{align}\label{H-infty}
		\|\mathcal{H}\|_{L^{\infty}(\Sn )}\leq O(\|K\circ\varphi_{P, t}-1\|_{L^{\infty}(\Sn )}).
		\end{align}
		
		Thanks to \eqref{H-infty} and \cite[Corollary 2.1]{jlxm}, we obtain
		\begin{align}\label{es-wp-1}
		\|v_{p}\|_{L^{\infty}(\Sn )}\leq O(\|K\circ\varphi_{P, t}-1\|_{L^{\infty}(\Sn )}).
		\end{align}
		Putting the above estimate and \eqref{es-wp-1} together, we find
		\begin{align}\label{Lamp-kp}
		|\Lambda_{p}|\leq C(n,\si)\|K\circ\varphi_{P, t}-K(P)\|_{L^{2}(\Sn )}.
		\end{align}	
		Using  Proposition \ref{lem3.6}, \eqref{es-wp-1} and $\|K-1\|_{L^{\infty}}\leq\va_{3}<\va_{1}$,	we have
		\begin{align}
		&|\lambda_{p}-\frac{P_{\si}(1)}{K(P)}|\notag\\
		\leq&
		C(n,\si,\va_1)
		\Big(\frac{\|w_{p}-1\|_{H^{\si}(\Sn)}+\|w_{p}-1\|_{L^{1}(\Sn )}
			+\dashint_{\Sn }|K\circ \varphi_{P, t}-K(P)|w_{p}^{\frac{2n}{n-2\si}} }
		{\dashint_{\Sn }K\circ \varphi_{P, t}w_{p}^{\frac{2n}{n-2\si}} }\Big)\notag \\
		\leq &C(n,\si,\va_1) \|K\circ \varphi_{P, t}-K(P)\|_{L^2(\Sn )}.\label{lamp-kp}
		\end{align}
		Then, by \eqref{lamp-kp} and \eqref{Lamp-kp}, we have
		\begin{align}
		\|(\lambda_{p} K\circ \varphi_{P, t}-\Lambda_{p}\cdot x)w_{p}^{\frac{n+2\si}{n-2\si}}
		-P_{\si}(1)\|_{L^{2}(\Sn )}
		\leq C(n,\si, \va_{1})\|K\circ \varphi_{P, t}-K(P) \|_{L^{2}(\Sn )}.\label{lampLamp}
		\end{align}
		Combining \eqref{eq-wp-2}, \eqref{lampLamp}, \eqref{psi-eig},	and  the spherical expansion of $w_{p}-1$, we arrive at
		\begin{align*}
		\|w_{p}-1\|_{H^{1}(\Sn )}^2\leq \int_{\Sn } (P_{\si }(w_p-1))^2\leq
		C(n,\si,\va_1) \|K\circ\varphi_{P, t}-K(P)\|_{L^{2}(\Sn )}.
		\end{align*}
		This justifies the  claim.
		
		\textbf{Step 3:} Complete the proof.	

From \eqref{pro-eq4} and \eqref{pro-eq5}, we know that for sufficiently large $t$,
		there exists $0<C_0<1$ such that
		\begin{align*}
		\mathscr{A}_{p}\cdot \mathscr{B}_{p}\geq (1-C_0)|\mathscr{B}_{p}|^2.
		\end{align*}
		As a consequence of the homotopy invariance property of the degree,
		\begin{align*}
		\mathrm{deg}(\mathscr{A}_{p}, B^{n+1}, 0)=
		\mathrm{deg}(\mathscr{B}_{p}, B^{n+1}, 0).
		\end{align*}
		
		We note that $\Lambda_{p}$ can also be computed more directly from the function $K$ as
		follows. In view of \eqref{eq-wp} and the Kazdan-Warner identity (see \cite{jlxm}), we have	
		\begin{align*}
		\int_{\Sn}\langle\nabla( \lambda_{p} K \circ \varphi_{P, t}-\Lambda_{p} \cdot x), \nabla x_i\rangle w_{p}^{\frac{2 n}{n-2\si}} =0, \quad 1 \leq i \leq n+1.
		\end{align*}
		It follows that, for $1 \leq  i \leq  n + 1$,
		\begin{align}\label{57}
		\sum_{j=1}^{n+1} \Lambda_{p}^{j}
		\int_{\Sn }\langle\nabla x_{j}, \nabla x_{i}\rangle w_{p}^{\frac{2 n}{n-2 \si }}=\lambda_{p} \int_{\Sn }\langle\nabla(K \circ \varphi_{P, t}), \nabla x_{i}\rangle w_{p}^{\frac{2 n}{n-2 \si }}.
		\end{align}
		Note that, as $\va_3\rightarrow 0$, by  Proposition \ref{lem3.6}, we have $w_p\rightarrow 1$ uniformly for small $\va_3$. This
		implies that the coefficient matrix on the left hand side of \eqref{57} is positive definite:	
		$$\Big(\int_{\Sn } \langle \nabla x_{i}, \nabla x_{j}\rangle w_{p}^{\frac{2n}{n-2\si }}\Big)_{1\leq i,j\leq n+1}>0.$$
		It follows that, for $t$   large with $ s=(t-1) / t$,
		\begin{align}
		\mathrm{deg}\left(\Lambda_{p}, B_{s}^{n+1}, 0\right)
		=\mathrm{deg}\left(\mathscr{A}_{p}, B^{n+1}, 0\right)
		=\mathrm{deg}\left(\mathscr{B}_{p}, B^{n+1}, 0\right),
		\end{align}
		where $B_{s}^{n+1}$ denotes the open ball in $\mathbb{R}^{n+1}$
		with centered at the origin and $s$ as the radius.	Therefore, $\Lambda_{p}$ has to have a zero inside $B^{n+1}$	which immediately implies that \eqref{main-eq} has at least one positive solution.
		
		Let $I(w,p)$ be as in \eqref{Ip},
		the same argument in \cite[Theorem 3.1]{JLXOn2015} can be applied to obtain that, for any $p_{0} \in B^{n+1}$,
		\begin{align*}
		&\partial_{p} I(w_{p_{0}}, p)|_{p=p_{0}}\\
		=&-\frac{n-2 \si }{n}\Big(\dashint_{\Sn } K (T^{-1}_{\varphi_{p_0}}w_{p_0})^{2 n /(n-2 \si )}\Big)^{\frac{2 \si -n}{n}} \partial_{p}\Big(\dashint_{\Sn } \Lambda_{p_{0}} \cdot \varphi_{p_{0}}^{-1} \circ \varphi_{P, t} w_{p_{0}}^{\frac{2 n}{n-2 \si }}\Big)\Big|_{p=p_{0}}.
		\end{align*}
		By Li \cite[Appendix A]{LPrescribing1995}, we get the matrix
		\begin{align*}
		\partial_{p}\Big(\dashint_{\Sn } \Lambda_{p_{0}} \cdot \varphi_{p_{0}}^{-1} \circ \varphi_{P, t} w_{p_{0}}^{\frac{2 n}{n-2 \si }}\Big)\Big|_{p=p_{0}}
		\end{align*}
		is invertible with positive determinant. Therefore, for
		$t$   large with $ s=(t-1) / t$ , we have
		\begin{align*}
		(-1)^{n+1} \mathrm{deg}(\Lambda_{p}, B_{s}^{n+1}, 0)
		=\mathrm{deg}(\partial_{p} I(w_{p}, p), B_{s}^{n+1}, 0).
		\end{align*}
		The rest of the proof of \eqref{prop-1} is similar to that in \cite[page 386]{LPrescribing1995}
		and we omit them here.
	\end{proof}
	
	\section{Proof of Theorem \ref{them2} and Corollary \ref{cor1}}\label{sec:4}
	In this section, we give the proof of Theorem \ref{them2} and Corollary \ref{cor1}.
	\begin{proof}[Proof of Theorem \ref{them2}] The proof is divided into three steps.
		
		\textbf{Step 1:} Proof of \eqref{th-bound}.
		Suppose the contrary that the solution $v$ to \eqref{main-eq}
		has at least one isolated simple blow up point.
		In the case of $\sharp \mathscr{K}_{n-2\si }^{-} \leq 1$,
		it follows from Theorem \ref{pro-lambda-exi}
		that $v$ has only one blow up point,
		and then  we  obtain from  \cite[Theorem 3.5]{jlxm}
		that there exists a constant $C>0$
		such that
		$$1/C<v<C\quad \text{ on }\, \Sn . $$
		
		In another case,  we assume that $v$ corresponding to \eqref{main-eq} blow up at
		$\{ q^{(1)},\ldots,q^{(k)}\}$ with $k\geq 2$.
		By Theorem \ref{pro-lambda-exi} we know that equation
		\begin{align*}
		M(q^{(1)},\ldots, q^{(k)}, \eta^{(1)},\ldots,\eta^{(k)})
		\begin{pmatrix}
		\lambda_1 \\
		\vdots \\
		\lambda_{k}
		\end{pmatrix}
		=0
		\end{align*}
		has at least one solution $\lambda_{1},\ldots,\lambda_{k}>0$.
		By Cramer's Rule we can deduce a contradiction.
		Therefore, in both cases we prove that \eqref{th-bound} holds.

		\textbf{Step 2:} We now prove the following claim: let $\va_{3}$ be as in
		Proposition \ref{prop-theorem3.1},
		and $\mathcal{N}_{1}$ be as in \eqref{nn1},
		then there exists a constant $\va_{4}>0$ such that, for $0\leq \mu\leq \va_{4}$, we have
		$\|K_{\mu}-1\|_{L^{\infty}(\Sn )}<\va_{3}$, where $K_{\mu}:= \mu K+(1-\mu)$.
		Furthermore, if $v$ is any solution to \eqref{main-eq} with $K=K_{\mu}$,
		and there exists $(w,p)\in \mathscr{S}_{0}\times B^{n+1}$ such that $ v=\pi(w,p)$,
		then $w\in \mathcal{N}_{1}$.
		The proof of the claim  is similar to that in \cite[page 1529]{JLXOn2015}, and
		we omit it here.

		\textbf{Step 3:} Proofs of \eqref{them-eq-1}
		and \eqref{them-eq-2}.
		Equation \eqref{them-eq-1} follows from \cite[Lemma 6.7]{LPrescribing1995}. Next we prove \eqref{them-eq-2}.		

		From the proof of Step 1,
		it is known that there exists some constant $C_{0}>1$ such that
		for all $\va_4\leq \mu<1$,
		$$1/C_0 <v_{\mu}<C_{0}, $$
		where $v_{\mu}$  is any solution of \eqref{main-eq} with $K=K_{\mu}$.
		
		For $\si \in[1,n/2)$, it follows from the homotopy property of the Leray Schauder
		degree and Proposition \ref{prop-theorem3.1} that
		\begin{align*}
		&\mathrm{deg}\Big(v-(P_{\si })^{-1} K v^{\frac{(n+2 \si )}{ (n-2 \si )}}, C^{2 \si +\alpha}(\Sn ) \cap\{1 / C_{0} \leq v_{\mu} \leq C_{0}\}, 0\Big) \\
		=&\mathrm{deg}\Big(v-(P_{\si })^{-1} K_{\varepsilon_{4}} v^{\frac{(n+2 \si )}{(n-2 \si )}}, C^{2 \si +\alpha}(\Sn ) \cap\{1 / C_{0} \leq v_{\mu} \leq C_{0}\}, 0\Big) \\
		=&(-1)^{n} \mathrm{deg}\Big(\int_{\Sn }
		K_{\varepsilon_{4}} \circ \varphi_{P, t}(x) x, B^{n+1}, 0\Big) \\
		=&(-1)^{n} \mathrm{deg}
		\Big(\int_{\Sn } K \circ \varphi_{P, t}(x) x, B^{n+1}, 0\Big).
		\end{align*}
		For $\si \in(0,1)$, Equation \eqref{them-eq-2} follows from \cite[Theorem 3.2]{JLXOn2015}.

		If
		\begin{align*}
		\mathrm{deg}\Big( \int_{\Sn } K\circ \varphi_{P,t}(x)x, B^{n+1},0\Big) \ne 0
		\end{align*}
		for large $t$, then \eqref{main-eq} has at least one solution.
	\end{proof}
	
	\begin{proof}[Proof of Corollary \ref{cor1}]
		Corollary \ref{cor1} follows from  Proposition \ref{prop-theorem3.1} and
		\cite[Lemma 6.7]{LPrescribing1995},
		the proof of Theorem \ref{them2} and \cite[Corollary 6.2]{LPrescribing1995}.
	\end{proof}

	\appendix
	
	\section{Appendix}\label{sec:A}

In this section, we review some results about the blow
up profiles for  integral equations obtained in Jin-Li-Xiong \cite{jlxm}.
For any $x\in \mathbb{R}^{n}$ and $r>0,$
${B}_{r}(x)$ denotes the ball in $\mathbb{R}^{n}$
with radius $r$ and center $x$, and $B_{r}:=B_{r}(0).$

Let $\Omega$ be a domain in $\mathbb{R}^{n}$ and $K_{i}$ are
	nonnegative bounded functions in $\mathbb{R}^{n}.$
	Let $\{\tau_{i}\}_{i=1}^{\infty}$ be a sequence of
	nonnegative constants satisfying $\lim _{i \rightarrow \infty} \tau_{i}=0$, and set
	$$
	p_{i}=\frac{n+2 \si }{n-2 \si }-\tau_{i}.
	$$
	Suppose that $0 \leq u_{i} \in L_{{loc}}^{\infty}(\mathbb{R}^{n})$
	satisfies the nonlinear integral equation
	\be\label{rie}
	u_{i}(x)=\int_{\mathbb{R}^{n}} \frac{K_{i}(y) u_{i}(y)^{p_{i}}}{|x-y|^{n-2 \si }}\, \ud y
	\quad \text { in } \,\Omega.
	\ee
	We assume that $K_{i} \in C^{1}(\Omega)$ $(K_{i}\in C^{1,1}(\Omega )$ if
	$\si \leq 1/2$) and, for some positive constants $A_{1}$ and $A_{2}$,
	\be\label{rk1}
	1 / A_{1} \leq K_{i}, \,
	\|K_{i}\|_{C^{1}(\Omega)} \leq A_{2},\,( \|K_{i}\|_{C^{1,1}(\Omega)} \leq A_{2}\, \text { if }\,\si  \leq \frac{1}{2}).
	\ee

	\begin{proposition}[Pohozaev type identity]\label{pro-pz}
		Let $u \geq 0$ in $\mathbb{R}^{n}$, and $u \in C(\overline{B}_{R})$ be a solution of
		$$
		u(x)=\int_{B_{R}} \frac{K(y) u(y)^{p}}{|x-y|^{n-2 \si }} \,\ud y+h_{R}(x),
		$$
		where $1<p \leq \frac{n+2 \si }{n-2 \si },$ and $h_{R}(x) \in C^{1}(B_{R}),$ $\nabla h_{R} \in L^{1}(B_{R}).$ Then
		\begin{align*}
		&\Big(\frac{n-2 \si }{2}-\frac{n}{p+1}\Big) \int_{B_{R}} K(x) u(x)^{p+1} \,\ud x
		-\frac{1}{p+1} \int_{B_{R}} x \nabla K(x) u(x)^{p+1} \,\ud x \\
		=& \frac{n-2 \si }{2} \int_{B_{R}} K(x) u(x)^{p} h_{R}(x) \,\ud x
		+\int_{B_{R}} x \nabla h_{R}(x) K(x) u(x)^{p} \,\ud x \\
		&-\frac{R}{p+1} \int_{\partial B_{R}} K(x) u(x)^{p+1} \, \ud s.
		\end{align*}
	\end{proposition}
	
	\begin{proposition}\label{pro-ibp-har}
		Suppose that $0 \leq u_{i} \in L_{{loc}}^{\infty}(\mathbb{R}^{n})$ satisfies \eqref{rie}
		with $K_{i}$ satisfying \eqref{rk1}. Suppose that $x_{i} \rightarrow 0$
		is an isolated blow up point of $\{u_{i}\}$, i.e., for some positive
		constants $A_{3}$ and $\bar{r}$ independent of $i$,
		$$
		|x-x_{i}|^{2 \si  /(p_{i}-1)}u_{i}(x) \leq A_{3}\quad \text { for all }\, x \in B_{\bar{r}} \subset \Omega.
		$$
		Then for any $0<r<\bar{r}/3$, we have the following Harnack inequality
		$$
		\sup _{B_{2r}(x_{i}) \backslash \overline{B_{r / 2}(x_{i})}} u_{i}
		\leq C \inf_{B_{2 r}(x_{i})\backslash \overline{B_{r/2}(x_{i})}} u_{i},
		$$
		where $C$ is a positive constant depending only on $\sup_{i}\|K_{i}\|_{L^{\infty}(B_{\bar{r}}(x_{i}))},
		n, \si , \bar{r}$ and $A_{3}.$
	\end{proposition}
	
	\begin{proposition}\label{pro-ibp-upb}
		Assume the hypotheses  in Proposition \ref{pro-ibp-har}.
		Then for every $R_{i} \rightarrow \infty$, $\varepsilon_{i} \rightarrow 0^{+},$
		we have, after passing to a subsequence (still denoted as $\{u_{i}\},$ $\{x_{i}\},$ etc.), that
		$$
		\|m_{i}^{-1} u_{i}(m_{i}^{-(p_{i}-1) /2\si }\cdot+x_{i})
		-(1+k_{i}|\cdot|^{2})^{(2 \si -n) / 2}\|_{C^{2}(B_{2 R_{i}}(0))} \leq \varepsilon_{i},
		$$
		$$
		r_{i}:=R_{i} m_{i}^{-(p_{i}-1) / 2 \si } \rightarrow 0  \quad \text { as }\, i \rightarrow \infty,
		$$
		where $ m_{i}:=u_{i}(x_{i})$ and $ k_{i}:=({K_{i}(x_{i}) \pi^{n/2}\Gamma(\si )}/{\Gamma(\frac{n}{2}+\si )})^{1/\si }.$
	\end{proposition}
	
	\begin{proposition}\label{pro-ibp-lowb}
		Under the  hypotheses  of Proposition \ref{pro-ibp-har},
		there exists a positive constant $C=C(n, \si , A_{1}, A_{2}, A_{3})$ such that,
		$$u_{i}(x) \geq C^{-1} m_{i}(1+k_{i} m_{i}^{(p_{i}-1) / \si }|x-x_{i}|^{2})^{(2 \si -n) / 2}\quad
		\text{ for all } \, |x-x_{i}| \leq 1.$$
		In particular, for any $e \in \mathbb{R}^{n},|e|=1$, we have
		$$
		u_{i}(x_{i}+e) \geq C^{-1} m_{i}^{-1+((n-2 \si ) / 2 \si ) \tau_{i}},
		$$
		where $\tau_{i}=(n+2 \si ) /(n-2 \si )-p_{i}$.
	\end{proposition}

	\begin{proposition}\label{pro-isbp-upb}
		Under the hypotheses of Proposition \ref{pro-ibp-har} with $\bar{r}=2,$
		and in addition that $x_{i} \rightarrow 0$ is also an isolated simple blow up point with constant $\rho,$
		we have
		$$
		\tau_{i}=O(u_{i}(x_{i})^{-c_{1}+o(1)})\quad \text{and}\quad
		u_{i}(x_{i})^{\tau_{i}}=1+o(1),
		$$
		where $c_{1}=\min \{2,2 /(n-2 \si )\}$.  Moreover,
		$$
		u_{i}(x) \leq C u_{i}^{-1}(x_{i})|x-x_{i}|^{2 \si -n} \quad \text { for all }\,|x-x_{i}| \leq 1.
		$$
	\end{proposition}
	
	\begin{proposition}\label{pro-isbp-lim}
		Under the hypotheses of Proposition \ref{pro-isbp-upb},
		let
		\begin{align*}
		T_{i}(x):=&
		u_{i}(x_{i})
		\int_{B_{1}(x_{i})} \frac{K_{i}(y) u_{i}(y)^{p_{i}}}{|x-y|^{n-2 \si }} \, \mathrm{d} y
		+
		u_{i}(x_{i}) \int_{\mathbb{R}^{n} \backslash B_{1}(x_{i})} \frac{K_{i}(y) u_{i}(y)^{p_{i}}}{|x-y|^{n-2 \si }}\, \mathrm{d} y\\
		=&:T_{i}^{\prime}(x)+T_{i}^{\prime \prime}(x).
		\end{align*}
		Then, after passing to a subsequence,
		$$
		T_{i}^{\prime}(x) \rightarrow a|x|^{2 \si -n}
		\quad \text { in } \, C_{loc}^{2}(B_{1} \backslash\{0\})
		$$
		and
		$$
		T_{i}^{\prime \prime}(x) \rightarrow h(x) \quad \text { in } \, C_{l o c}^{2}(B_{1})
		$$
		for some $h(x) \in C^{2}(B_{2})$, where
		$$
		a=\Big(\frac{\pi^{n / 2} \Gamma(\si )}
		{\Gamma(\frac{n}{2}+\si )}\Big)^{-\frac{n}{2 \si }} \int_{\mathbb{R}^{n}}\Big(\frac{1}{1+|y|^{2}}\Big)^{\frac{n+2 \si }{2}} \mathrm{~d} y \lim _{i \rightarrow \infty} K_{i}(0)^{\frac{2 \si -n}{2 \si }}.
		$$
		Consequently, we have
		$$
		u_{i}(x_{i}) u_{i}(x) \rightarrow a|x|^{2 \si -n}
		+b(x) \quad\text { in } \, C_{l o c}^{2}(B_{1} \backslash\{0\}).
		$$
	\end{proposition}
	
	\begin{proposition}\label{pro-esti}
		Under the hypotheses of Proposition \ref{pro-ibp-har}, we have
		$$
		\int_{|y-y_{i}| \leq r_{i}}|y-y_{i}|^{s}
		u_{i}(y)^{p_{i}+1}=
		\begin{cases}
		O(u_{i}(y_{i})^{-2 s /(n-2 \si )}), & -n<s<n, \\ O(u_{i}(y_{i})^{-2 n /(n-2 \si )} \log u_{i}(y_{i})), & s=n, \\ o(u_{i}(y_{i})^{-2 n /(n-2 \si )}),
		& s>n,
		\end{cases}
		$$
		and
		$$
		\int_{r_{i}<|y-y_{i}| \leq 1}|y-y_{i}|^{s} u_{i}(y)^{p_{i}+1}=
		\begin{cases}
		o(u_{i}(y_{i})^{-2 s /(n-2 \si )}), & -n<s<n, \\ O
		(u_{i}(y_{i})^{-2 n /(n-2 \si )} \log u_{i}
		(y_{i})), & s=n, \\ O(u_{i}(y_{i})^{-2 n /(n-2\si  )}), & s>n .\end{cases}
		$$
	\end{proposition}

	\bigskip
	
	\noindent Yan Li
	
	\noindent School of Mathematical Sciences, Beijing Normal University\\ Beijing 100875, China\\
	Email: \textsf{yanli@mail.bnu.edu.cn}
	
	\medskip
	
	\noindent Zhongwei Tang
	
	\noindent School of Mathematical Sciences, Beijing Normal University\\ Beijing 100875, China\\
	Email: \textsf{tangzw@bnu.edu.cn}
	
	\medskip
	\noindent Heming Wang
	
	\noindent School of Mathematical Sciences, Beijing Normal University\\
	Beijing 100875, China\\[1mm]
	Email: \textsf{hmw@mail.bnu.edu.cn}

	\medskip
	\noindent Ning Zhou
	
	\noindent School of Mathematical Sciences, Beijing Normal University \\Beijing 100875, China\\
	Email: \textsf{nzhou@mail.bnu.edu.cn}

\end{document}